\documentclass[11pt]{article}

\usepackage{amsmath,amsfonts,amsthm,latexsym}

\theoremstyle{plain}
\newtheorem{lemma}{Lemma}[section]
\newtheorem{theorem}[lemma]{Theorem}
\newtheorem{theo}[lemma]{Theorem}
\newtheorem{corollary}[lemma]{Corollary}
\newtheorem{coro}[lemma]{Corollary}
\newtheorem{proposition}[lemma]{Proposition}
\newtheorem{prop}[lemma]{Proposition}
\newtheorem{claim}[lemma]{Claim}
\newtheorem*{corollary*}{Corollary}
\newtheorem*{theorem*}{Theorem}
\newtheorem*{lemma*}{Lemma}
\newtheorem*{claim*}{Claim}
\newtheorem*{proposition*}{Proposition}
\newtheorem* {theorem A}{Theorem A}
\newtheorem* {theorem A'}{Theorem A'}
\newtheorem* {theorem B}{Theorem B}
\newtheorem* {theorem C}{Theorem C}
\newtheorem* {theorem C'}{Theorem C'}

\newtheorem* {conjecture palis}{Conjecture (Palis)}
\newtheorem* {conjecture A}{Conjecture A}
\newtheorem* {corollary A}{Corollary A}
\newtheorem* {corollary A1}{Corollary A1}
\newtheorem* {corollary A2}{Corollary A2}
\newtheorem* {corollary B}{Corollary B}
\newtheorem* {corollary B2}{Corollary B2}
\newtheorem* {franks' lemma}{Franks' Lemma}
\newtheorem* {sigmund}{Theorem (Sigmund, 1970)}
\newtheorem* {theo1'}{Theorem \ref{theo1}'}
\newtheorem* {theo2'}{Theorem \ref{theo2}'}

\newtheorem* {main lemma}{Main Lemma}
\newtheorem* {main lemma'}{Main Lemma'}
\newtheorem* {connecting lemma}{Connecting Lemma}
\newtheorem* {topological lemma}{Topological Lemma}

\newtheorem* {ergodicdec}{Ergodic Decomposition Theorem}

\newtheorem* {ergodicclos}{Ergodic Closing Lemma}

\newtheorem* {semi}{Semicontinuity Lemma}

\newtheorem* {proposition A}{Proposition A}
\newtheorem* {proposition A'}{Proposition A'}
\newtheorem* {propositionbi'}{\textbf{Proposition \ref{propb.i}'}}
\newtheorem* {propositionbii'}{\textbf{Proposition \ref{propb.ii}'}}
\newtheorem* {propositionbiii'}{\textbf{Proposition \ref{propb.iii}'}}
\newtheorem* {propositionbiv'}{\textbf{Proposition \ref{propb.iv}'}}
\newtheorem* {theorem 0}{Theorem 0}
\newtheorem* {proof of theorem A}{\emph{Proof of Theorem A}}
\newtheorem* {proof of theorem B}{\emph{Proof of Theorem B}}
\newtheorem {question}{Question}

\newtheorem{definition}[lemma]{Definition}
\theoremstyle{remark}
\newtheorem{remark}[lemma]{Remark}

\def\trans{\mbox{~$|$\hspace{ -.46em}$\cap$}~}

\def\cO{\mathcal{O}}

\def\RR{\mathbb{R}}

\def\ZZ{\mathbb{Z}}
\def\NN{\mathbb{N}}

\def\eps{\varepsilon}
\def\empty{\emptyset}
\def\diff{\mathrm{Diff}^1(M)}

\def\CS{\mathcal{S}}

\def\cC{\mathcal{C}}

\def\cR{\mathcal{R}}
\def\cD{\mathcal{D}}
\def\cS{\mathcal{S}}
\def\cU{\mathcal{U}}
\def\cZ{\mathcal{Z}}
\def\cV{\mathcal{V}}
\def\cW{\mathcal{W}}
\def\cG{\mathcal{G}}
\def\cO{\mathcal{O}}
\def\cA{\mathcal{A}}
\def\cM{\mathcal{M}}
\def\cP{\mathcal{P}}
\def\cK{\mathcal{K}}
\def\mfla{\mathcal{M}_f(\La)}
\def\mfga{\mathcal{M}_f(\Gamma)}
\def\mflerg{\mathcal{M}_f^{erg}(\La)}
\def\mfmerg{\mathcal{M}_f^{erg}(M)}

\def\mfhp{\mathcal{M}_f(H(p))}
\def\pfla{\mathcal{P}_f(\La)}

\def\perfla{Per_f(\La)}

\def\mfm{\mathcal{M}_f(M)}
\def\pfm{\mathcal{P}_f(M)}

\def\mm{\mathcal{M}(M)}
\def\over{\overline}

\def\La{\Lambda}
\def\la{\lambda}

\def\te{\tilde{E}}
\def\tf{\tilde{F}}
\def\cms{\cC\cM(\Sigma_p)}
\def\mx{\mathcal{M}_A(X)}
\def\loc{\text{loc}}
\def\cC{\mathcal{C}}
\def\supp{\operatorname{Supp}}
\def\diam{\operatorname{Diam}}
\def\interior{\operatorname{Int}}

\title{Nonuniform hyperbolicity for $C^1$-generic diffeomorphisms}
\author{Flavio Abdenur\footnote{Partially supported by a CNPq/Brazil research grant.}, Christian Bonatti, and Sylvain Crovisier}

\begin{document}

\maketitle

\begin{abstract}
We study the ergodic theory of non-conservative  $C^1$-generic diffeomorphisms. First, we show that
homoclinic classes of arbitrary diffeomorphisms exhibit ergodic
measures whose supports coincide with the homoclinic class. Second,
we show that generic (for the weak topology) ergodic
measures of $C^1$-generic diffeomorphisms are nonuniformly hyperbolic:
they exhibit no zero Lyapunov exponents. Third, we extend a theorem
by Sigmund on hyperbolic basic sets: every isolated transitive set $\La$ of any
$C^1$-generic diffeomorphism $f$ exhibits many ergodic hyperbolic measures
whose supports coincide with the whole set $\La.$

In addition, confirming a claim made by R. Ma\~n\'e in $1982$, we
show that hyperbolic measures whose Oseledets splittings are
dominated satisfy Pesin's Stable Manifold Theorem, even if the
diffeomorphism is only $C^1$.

\bigskip


\medskip

\end{abstract}

\section{Introduction}

In his address to the 1982 ICM, R. Ma\~n\'e \cite{M3} speculated on
the ergodic properties of $C^1$-generic diffeomorphisms. He divided
his discussion into two parts, the first dealing with
non-conservative (i.e. ``dissipative'') diffeomorphisms, the second
with conservative diffeomorphisms.

In the first part, drawing inspiration from the work of K. Sigmund
\cite{Sig} on generic measures supported on basic sets of Axiom A
diffeomorphisms, Ma\~n\'e first used his Ergodic Closing Lemma
\cite{M2} to show that ergodic measures of generic diffeomorphisms
are approached in the weak topology by measures associated to
periodic orbits (this is item (i) of Theorem \ref{theo3} of this
paper; we include a detailed proof, since Ma\~n\'e did not). He then
went on to prove that the Oseledets splittings of generic ergodic
measures\footnote{The space $\mfmerg$ of ergodic measures of a
diffeomorphism $f$ is a Baire space when endowed with the weak
topology, so that its residual subsets are dense; see Subsection
\ref{b.isection}.} of generic diffeomorphisms are in fact uniformly
dominated, and to claim that such conditions -- uniformly dominated
Oseledets splittings -- together with nonuniform hyperbolicity are
sufficient to guarantee the existence of smooth local stable
manifolds at $\mu$-a.e. point, as in Pesin's Stable Manifold Theorem
\cite{Pe}.

In the second part, discussing the case of conservative
diffeomorphisms, he stated a $C^1$-generic dichotomy between (some
form of) hyperbolicity and an abundance of orbits with zero Lyapunov
exponents. In the two-dimensional setting this reduced to a
dichotomy between Anosov diffeomorphisms and those having zero
exponents at almost every orbit. Ma\~n\'e never published a proof of
this dichotomy.

\medskip

For conservative diffeomorphisms much progress has been made. The
generic dichotomy between hyperbolicity and zero Lyapunov exponents
for surface diffeomorphisms, in particular, was proven by Bochi
\cite{Boc1} in 2000, later extended to higher dimensions by Bochi
and Viana \cite{BocV}, and finally settled in the original (symplectic, in arbitrary dimension) statement of Ma\~n\'e by Bochi \cite{Boc2} in 2007. Many other important results have been
obtained for $C^1$-generic conservative diffeomorphisms, see for
instance  \cite{ArBC,DW,HoT}.

By contrast, there was for a long time after Ma\~n\'e's address
little progress towards the development of the ergodic theory for
$C^1$-generic dissipative diffeomorphisms. This is in our view due
to the two following obstacles:

\begin{itemize}
\item
\textbf{Obstacle 1: The Absence of Natural Invariant Measures.}
Conservative diffeomorphisms are endowed with a natural
invariant measure, namely the volume that is preserved. In the
dissipative context, hyperbolic basic sets are endowed with some
very interesting invariant measures, such as the measure of maximal
entropy (see \cite{Bow}), or, in the case of hyperbolic attractors,
the Sinai-Ruelle-Bowen measure (see for instance \cite{Ru}). In the
case of $C^1$-generic dissipative diffeomorphisms, however, it is
difficult to guarantee the existence of measures describing most of the underlying dynamics.
For instance, Avila-Bochi \cite{AvBoc} have recently shown that
$C^1$-generic maps do not admit absolutely continuous
invariant measures.
\item
\textbf{Obstacle 2: The $C^1$-Generic Lack of $C^2$-Regularity.} For
much of differentiable ergodic theory the hypothesis of $C^1$
differentiability is insufficient; higher regularity, usually $C^2$
but at least $C^1$+H\"{o}lder, is required. This is the case for
instance of Pesin's Stable Manifold Theorem \cite{Pe} for
nonuniformly hyperbolic dynamics \footnote{Obstacle 2,
unlike Obstacle 1, is of course also a problem in the conservative
setting.}.

\end{itemize}

\medskip

The aim of this paper is to realize \emph{some of
Ma\~n\'e's vision of an ergodic theory for non-conservative
$C^1$-generic diffeomorphisms.} Some of our results confirm claims
made without proof by Ma\~n\'e; others extend Sigmund's work to the
nonhyperbolic $C^1$-generic setting; and still others go beyond the
scope of both of these previous works. In any case, our results
begin to tackle both of the aforementioned obstacles to a generic
ergodic theory. We hope that this work will help the development
of a rich ergodic theory for $C^1$-generic
dissipative diffeomorphisms.

\medskip

Our starting point is the generic \emph{geometric} theory for
dissipative diffeomorphisms, that is, the study from the $C^1$-generic viewpoint of non-statistical
properties: transitivity,  existence of dominated
splittings,  Newhouse phenomenon (coexistence of an infinite
number of periodic sinks or sources)\dots
There has been, especially
since the mid-$90$'s, an explosion of important generic geometric
results, thanks largely to Hayashi's Connecting Lemma \cite{H}.
It turns out, however, that many of
these tools -- especially from \cite{ABCDW}, \cite{BDP}, and
\cite{BDPR} -- are also useful for the study of generic ergodic
problems. Our results on generic ergodic theory follow largely from
the combined use of these geometric tools with techniques by Sigmund
and Ma\~n\'e.
\medskip

Some of our results hold for every diffeomorphism, some require a $C^1$-generic assumption.
We can group them into three types:

\begin{itemize}
\item[a)]
\textbf{Approximations by Periodic Measures.} A classical consequence of Ma\~n\'e ergodic
closing lemma \cite{M2} is that, for $C^1$-generic diffeomorphisms, every invariant measure is the weak limit of
a convex sum of dirac measures along periodic orbits.
We propose some variation on this statement, for instance:

\emph{If $f$ is a $C^1$-generic diffeomorphism then}
\begin{itemize}
\item \emph{any ergodic measure $\mu$  is
the weak and Hausdorff limit of periodic measures whose Lyapunov
exponents converge to those of $\mu$} ~(Theorem~\ref{theo3});
\item \emph{any (non necessarily ergodic) measure supported on
an isolated transitive set $\La$ is the weak limit of periodic
measures supported on $\La$} ~(Theorem~\ref{theo4}~part (a)).
\end{itemize}

The idea  is to show that, analogously with
what occurs from the ``geometric'' viewpoint with Pugh's General
Density Theorem \cite{Pu1}, generically hyperbolic periodic measures
are abundant (e.g., dense) among ergodic measures, and so provide a
robust skeleton for studying the space of invariant measures.

\item[b)]
\textbf{Geometric Properties of Invariant Measures.} Some of our
results deal with the geometric and topological aspects of the
invariant measures, such as the sizes of their supports, their
Lyapunov exponents and corresponding Lyapunov spaces, and the
structure of their stable and unstable sets. For instance:
\begin{itemize}
\item \emph{Let $\La$ be an isolated transitive set of a $C^1$-generic diffeomorphism $f$. Then every generic
measure with support contained in $\Lambda$ is ergodic, has no zero Lyapunov exponents (i.e. is nonuniformly hyperbolic) and its support is equal to $\Lambda$}
(Theorem~\ref{theo4} ~part ( b)).
\item \emph{Let $\mu$ be an ergodic measure without zero Lyapunov exponent,
and whose support admits a dominated splitting
corresponding to the stable/unstable spaces of $\mu$.
Then there exists stable and unstable manifolds a
$\mu$-almost every point}  (Theorem \ref{theo5}).
\end{itemize}

\item[c)]
\textbf{Ergodic Properties of Invariant Measures.} Finally, many of
our results deal with ``statistical'' properties such as ergodicity
and entropy of the invariant measures. For instance:
\begin{itemize}
\item
\emph{Any homoclinic class coincides with the support of an ergodic measure with zero entropy}
(Theorem \ref{theo1}).
\end{itemize}

\end{itemize}

Some of our results may admit extensions to or analogues in
the conservative setting, but we have not explored this direction. 
\section{Preliminaries}

\subsection{General definitions}
Given a compact boundaryless $d$-dimensional manifold $M$, denote by
$\diff$ the space of $C^1$ diffeomorphisms of $M$
endowed with the usual $C^1$ topology.

\smallskip

Given a diffeomorphism $f \in \diff$, a point $x \in M$, and a constant
$\varepsilon > 0$, then the \emph{stable set} of $x$ is
$$W^s(x) := \{y \in M : d(f^k(x), f^k(y)) \to 0 \text{ as } k \to +\infty \}$$
and the \emph{$\varepsilon$-local stable set} of $x$ is
$$W^s_{\varepsilon}(x) := \{y \in W^s(x) : d(f^k(x), f^k(y)) \leq \varepsilon \text{ for every } k \in \NN\}.$$
The \emph{unstable set} $W^u(x)$ and the \emph{$\varepsilon$-local unstable set} $W^s_{\varepsilon}(x)$ are defined analogously.

\smallskip

Given $f \in \diff$, a compact $f$-invariant set $\La$ is
\emph{isolated} if there is some
neighborhood $U$ of $\La$ in $M$ such that $$\La = \bigcap_{k \in
\ZZ} f^k(U).$$

A compact $f$-invariant set $\La$ is \emph{transitive} if there
is some $x \in \La$ whose forward orbit
is dense in $\La$. A transitive
set $\La$ is \emph{trivial} if it consists of a periodic orbit.

We denote by $\cO(p)$ the orbit of a periodic point $p$ and by $\Pi(p)$ its period.

For $A\in GL(\RR,d)$ we denote by $m(A)=\|A^{-1}\|^{-1}$
its minimal dilatation.

\subsection{Homoclinic classes}

The Spectral Decomposition Theorem splits the nonwandering set
of any Axiom A diffeomorphism into \emph{basic sets} which are pairwise disjoint isolated
transitive sets. They are the homoclinic classes of periodic orbits. This
notion of homoclinic class can be defined in a more general setting:

\begin{definition}
Let $\cO(p)$ be a hyperbolic periodic orbit of $f \in \diff$. Then
\begin{itemize}
\item
the \emph{homoclinic class} of $\cO(p)$ is the set
$$H(\cO(p)) := \over{W^s(\cO(p)) \trans W^ u(\cO(p))};$$
\item
given an open set $V$ containing $\cO(p)$, the \emph{homoclinic
class of $\cO(p)$ relative to $V$} is the set
$$H_V(\cO(p)) := \over{\{x \in W^s(\cO(p)) \trans W^ u(\cO(p)) : \cO(x) \subset V\}}.$$
\end{itemize}
\end{definition}
Although the homoclinic class is
associated to the periodic orbit $\cO(p)$ of $p$, we write sometimes $H(p)$
instead of $H(\cO(p))$.

Relative homoclinic classes like full homoclinic
classes are compact transitive sets with dense subsets of periodic
orbits. There is another characterization of homoclinic classes:
\begin{definition}
Two hyperbolic periodic points $p$ and $q$ having the same stable dimension
are \emph{homoclinically related} if
$$W^s(\cO(p)) \trans W^u(\cO(q)) \neq \empty \text{ and } W^u(\cO(p)) \trans W^s(\cO(q))
\neq \empty$$
\end{definition}
If we define $\Sigma_p$ as the set of hyperbolic periodic points that are
homocliically related to $p$, then $\Sigma_p$ is $f$-invariant and
its closure coincides with $H(p)$.

In the relative case in an open set $V$ we denote by $\Sigma_{V,p}$ the set of hyperbolic periodic points
whose orbit is contained in $V$ and which are homoclinically related with $p$ by orbits contained in $V$.
Once more $H_V(p)$ is the closure of $\Sigma_{V,p}$.

\subsection{Invariant measures and nonuniform hyperbolicity}

The statements of many of our results involve two different types of
weak hyperbolicity: \emph{nonuniform hyperbolicity} and
\emph{dominated splittings}. We now recall the first of these two notions.

\medskip

\noindent -- The support of a measure $\mu$ is denoted by $\supp(\mu)$.
Given $\La$ a compact $f$-invariant set of some $f \in \diff$, set
$$\mfla := \{\mu: \mu \text{ is an $f$-invariant Borel
probability on $M$ such that $\supp(\mu)\subset \La$}\},$$
endowed with the weak topology.
Then, $\mfla$ is a compact metric space hence a
Baire space.

\medskip

\noindent -- We denote by $\mflerg$ the set of ergodic measures $\mu\in \mfla$.
This set is a $G_\delta$ subset of $\mfla$ (see Proposition~\ref{propb.i}), and
hence is a Baire space.
\medskip

\noindent -- Given $\gamma$ a periodic orbit of $f \in \diff$,
its associated \emph{periodic measure} $\mu_{\gamma}$
is defined by
$$\mu_{\gamma} := \frac{1}{\#\gamma} \sum_{p\in\gamma} \, \delta_{p}.$$
Given $\La$ a compact $f$-invariant set of some $f \in \diff$, set
$$\pfla := \{\mu_\gamma: \gamma \text{ is a periodic orbit in }\La \} \subset\mfm,$$
$$\perfla := \{p: p \text{ is a periodic point in }\La \}\subset M.$$
\medskip

\noindent --  Given any ergodic invariant probability
$\mu$ of a diffeomorphism of a compact manifold of dimension $d$
the \emph{Lyapunov vector of $\mu$}  denoted by  $L(\mu)\in\RR^d$
is the $d$-uple of the Lyapunov exponents of $\mu$, with
multiplicity,  endowed with an increasing order.
\medskip

\noindent An ergodic  measure $\mu \in \mfmerg$ is \emph{nonuniformly hyperbolic} if the
Lyapunov exponents of $\mu$-a.e. $x \in M$ are all non-zero. The
\emph{index} of a nonuniformly hyperbolic  measure $\mu$ is the sum of
the dimensions of Lyapunov spaces corresponding to its negative
exponents.
\medskip

\noindent A measure $\mu \in \mfm$ is \emph{uniformly hyperbolic} if
$\supp(\mu)$ is a hyperbolic set.
\medskip

\noindent -- Given a nonuniformly hyperbolic measure $\mu$ then its
\emph{hyperbolic Oseledets splitting}, defined at $\mu$-a.e. $x$, is
the $Df$-invariant splitting given by
$$\te^s(x) := \bigoplus_{\lambda_x < 0} \te(\lambda_x) \text{  and  } \te^u(x) := \bigoplus_{\lambda_x > 0} \te(\lambda_x),$$
where $\te(\lambda_x)$ is the Lyapunov space
corresponding to the Lyapunov exponent $\lambda_x$ at $x$.
\medskip

\noindent -- A point $x\in M$ is called \emph{irregular for positive
iterations} (or shortly \emph{irregular$^+$}) if there is a
continuous function $\psi\colon M\to \RR$ such that the sequence
$\frac 1n\sum_{t=0}^{n-1} \psi(f^t(x))$ is not convergent. A point
$x$ is \emph{Lyapunov irregular$^+$} if the Lyapunov exponents of
$x$ are not well-defined for positive iteration. \emph{Irregular$^-$} and \emph{Lyapunov irregular$^-$} points are defined analogously, considering negative iterates instead.

A point is \emph{regular} if it is regular$^+$ and regular$^-$ and
if furthermore the positive and negative average of any given
continuous function converge to the same limit.

\subsection{Dominated splitting}
We recall the definition and some properties of dominated splittings
(see~\cite[Appendix B]{BDV}).

\medskip

A $Df$-invariant splitting $T_\La M=E \oplus F$ of the tangent bundle over an
$f$-invariant set $\Lambda$ is \emph{dominated} if there exists $N\geq 1$
such that given any $x \in \La$, any unitary vectors $v \in E(x)$
and $w \in F(x)$, then
$${\|D_xf^N(v)\|}\leq \frac 1 2 {\|D_xf^N(w)\|}.$$
This will be denoted by $E\oplus_< F$.

More generally, a $Df$-invariant splitting $E_1 \oplus_< \ldots \oplus_<
E_t$ of the tangent bundle $T_{\La}M$ is a \emph{dominated
splitting} if given any $\ell \in \{1, \ldots, t - 1\}$ then the
splitting
$$(E_1 \oplus \ldots \oplus E_{\ell}) \oplus (E_{\ell + 1} \oplus \ldots \oplus E_t)$$
is dominated. A dominated splitting is \emph{non-trivial} if contains at least two
non-empty bundles.

\medskip

If an
invariant set $\La$ admits a dominated splitting $E_1 \oplus_<
\ldots \oplus_< E_t$, then:
\begin{itemize}

\item[a)]
the splitting $E_1(x) \oplus_< \ldots \oplus_< E_t(x)$ varies
continuously with the point $x \in \La$;

\item[b)]
the splitting $E_1 \oplus \ldots \oplus E_t$ extends to a dominated
splitting (also denoted by $E_1 \oplus_< \ldots \oplus_< E_t$) over the
closure $\over{\La}$ of $\La$;

\item[c)]
there is a neighborhood $V$ of $\over{\La}$ such that every
$f$-invariant subset $\Upsilon$ of $V$ admits a dominated splitting
$E_1' \oplus_< \ldots \oplus_< E_t'$  with $dim(E_i') = dim(E_i)$ for
each $i \in \{1, \ldots, t\}$.

\end{itemize}

There always exists a (unique) \emph{finest dominated splitting}
$F_1 \oplus_< \ldots \oplus_< F_k$ over $T_{\La}M$, characterized by
the following property: given
any dominated splitting $E' \oplus_< F'$ over $\La$ then there is
some $\ell \in \{1, \ldots, k - 1\}$ such that
$$E' = F_1 \oplus_< \ldots \oplus_< F_{\ell} \text{ and } F' = F_{\ell + 1} \oplus_< \ldots \oplus_< F_k.$$
That is, the finest dominated splitting $F_1 \oplus_< \ldots \oplus_< F_k$ is
\emph{minimal} in the sense that every dominated splitting over
$\La$ can be obtained by bunching together bundles of the finest
dominated splitting.
Equivalently, each of the bundles $F_i$ of the finest dominated
splitting is \emph{indecomposable}, in the sense that there exist no
subbundles $F_i^1$ and $F_i^2$ such that $F_i = F_i^1 \oplus F_i^2$
and
$$F_1 \oplus \ldots \oplus F_{i - 1} \oplus F_i^1 \oplus F_i^2 \oplus F_{i + 1} \oplus \ldots \oplus F_k$$
\noindent is a dominated splitting. Roughly speaking: ``there is no
domination within each $F_i$''.

The finest dominated splitting  ``separates Lyapunov
exponents''. That is, given $\mu \in \mfga$ an ergodic measure with
Oseledets splitting
$\te_1 \oplus \ldots \oplus \te_s$
and corresponding Lyapunov exponents $\la_1 < \la_2 <
\ldots < \la_s$ defined at $\mu$-a.e. $x$, then there are numbers $0
= j_0 < j_1 < j_2 < \ldots j_k = s$ such that for each $i \in \{1,
\ldots, k\}$
$$\bigoplus_{j_{i - 1} \, < m \; \leq j_i} \te_m(x) = F_i(x)$$
\noindent at $\mu$-a.e. $x$, where the $F_i$ are the bundles of the finest dominated splitting.
In other words, the bundles of the finest
dominated splitting can be written as sums of the Lyapunov spaces of
the increasing Lyapunov exponents of $\mu$. So we speak of the Lyapunov spaces  and  of the
Lyapunov exponents ``inside'' each bundle $F_i$. We denote by $L|_F(\mu)$
 the set of Lyapunov exponents of $\mu$  inside the bundle $F$; likewise,
given a Lyapunov-regular point $x \in \La$, we denote by $L|_F(x)$ the
set of Lyapunov exponents of $x$ inside $F$.

\subsection{Semicontinuity and genericity}\label{ss.semi}
Given $Y$ a compact metric space, we denote by $\cK(Y)$ the space of
compact subsets of $Y$ endowed with the \emph{Hausdorff distance}:
given two non-empty sets $K_1, K_2 \in \cK(Y)$, set
$$d_H(K_1, K_2) := \inf \{\eps>0: B_{\eps}(K_1) \supset K_2 \text{ and } B_{\eps}(K_2) \supset K_1\},$$

\noindent where $B_{\eps}(K)$ denotes the $\eps$-ball centered on
the set $K$. (The distance from the empty set to any non-empty set
is by convention equal to $\diam(Y)$.)

Then the space $(\cK(Y), d_H)$ is itself a compact (and hence a Baire)
metric space.

\begin{definition}
Given  a topological space $X$ and a compact metric space $Y$, a map $\Phi: X \rightarrow \cK(Y)$ is
\begin{itemize}
\item \emph{lower-semicontinuous}  at $x \in X$ if for any open $V \subset Y$ with $V \cap \Phi(x)
\neq \empty$,  there is a
neighborhood $U$ of $x$ in $X$ such that $V \cap \Phi(x') \neq
\empty$  for every $x' \in U$;
\item
\emph{upper-semicontinuous} at $x \in X$ if for any open $V \subset Y$ containing $\Phi(x)$,  there is a
neighborhood $U$ of $x$ in $X$ such that $V $ contains $\Phi(x')$ for every $x' \in U$;

\item  \emph{lower-semicontinuous} (resp,
\emph{upper-semicontinuous}) if it is lower-semicontinuous (resp,
upper-semicontinuous) at every $x \in X$.
\end{itemize}
\end{definition}

Now, we can state a result
from general topology (see for instance \cite{Kur}) which is one of
the keys to most of the genericity arguments in this paper:

\begin{semi}
Given $X$ a Baire space, $Y$ a compact metric space, and
$\Phi: X \rightarrow \cK(Y)$
\noindent a lower-semicontinuous (resp, upper-semicontinuous) map,
then there is a residual subset $\cR$ of $X$ which consists of
continuity points of $\Phi$.
\end{semi}

\begin{remark}
In this paper $X$ is usually either $\diff$ (with the $C^1$
topology) or else $\mfm$ (with the weak topology), while $Y$ is
usually $M$ or else $\mm$.
\end{remark}




\bigskip

In a Baire space, a set is \emph{residual} if it contains a countable
intersection of dense open sets. We establish a convention: the phrases ``generic
diffeomorphisms $f$ (resp., measures $\mu$) satisfy...'' and ``every
generic diffeomorphism $f$ (resp., measure $\mu$) satisfies...''
should be read as ``there exists a residual subset $\cR$ of $\diff$
(resp., of $\mfla$) such that every $f \in \cR$ (resp., every $\mu
\in \cR$) satisfies...'' 
\section{The Main Results} \label{resultssection}

\subsection{Homoclinic classes admit  ergodic measures with full support}
\begin{theorem} \label{theo1}
Let $H(p)$ be a relative homoclinic class of
a diffeomorphism $f \in \diff$. Then there is a measure $\mu \in
\mfhp$ which

\begin{itemize}
\item[i)]
is ergodic;

\item[ii)]
has ``full support'': $\supp(\mu) = H(p)$;

\item[iii)]
has zero entropy: $h_\mu(f) = 0$.

\end{itemize}
\end{theorem}

So any homoclinic class of \emph{any} diffeomorphism exhibits
at least one ergodic measure with full support. Theorem \ref{theo1}
is in fact a corollary of Theorem~\ref{theo1}'  stated in Section~\ref{theo1section}.

\begin{remark}
\begin{itemize}
\item One intriguing consequence of Theorem \ref{theo1} is this: given $f$
a $C^1$-generic \emph{conservative} diffeomorphism, then $f$ admits
at least one ergodic measure $\mu$ whose support coincides with all
of $M$. This follows from Theorem \ref{theo1} and the
fact that for
$C^1$-generic conservative diffeomorphisms the manifold $M$
is a homoclinic class (see \cite{BC}).

\item We think furthermore that the
($f$-invariant) volume  $m$ is approached in the
weak topology by ergodic measures with full support $\mu$;
we have not checked this completely,
the missing ingredient is a conservative version of the Transition Property Lemma in
Subsection \ref{mlemmaingredients}.
\end{itemize}
\end{remark}

\subsection{Generic ergodic measure of $C^1$-generic diffeomorphisms}

Methods similar to those used in the proof of Theorem \ref{theo1}'
yield an analogous result in the wider space of ergodic measures:

\begin{theorem}\label{theo2}
Given a $C^1$-generic diffeomorphism $f$ then every generic measure $\mu$
in $\mfmerg$

\begin{itemize}
\item[i)]
has zero entropy: $h_\mu(f) = 0$;

\item[ii)]
is nonuniformly hyperbolic and its Oseledets splitting
$\te_1 \oplus \ldots \oplus \te_k$ is dominated.
\end{itemize}

\end{theorem}

In Theorem~\ref{theo2} the domination of the  Oseledets splitting is due to Ma\~n\'e  \cite{M3}.

\begin{remark}
The splitting above is trivial when $\mu$
is supported on a periodic sink or source.
\end{remark}

\subsection{Isolated transitive sets of $C^1$-generic diffeomorphisms}

Isolated transitive sets are natural generalizations of
hyperbolic basic sets. Bonatti-Diaz \cite{BD2} used Hayashi's
Connecting Lemma \cite{H} to show that every isolated transitive set
of a $C^1$-generic diffeomorphisms is a relative homoclinic class
(see also~\cite{Ab1}).
Though at this point it is not known whether every generic
diffeomorphism exhibits some isolated transitive set, there are several
examples of locally generic diffeomorphisms having some non-hyperbolic isolated transitive sets,
for instance nonhyperbolic robustly transitive sets and
diffeomorphisms.

Theorem~\ref{theo4} below presents a overview of $C^1$-generic properties
satisfied by measures contained in an isolated transitive set.
\begin{theorem}\label{theo4}
Let $\La$ be an isolated non-trivial transitive set of a $C^1$-generic
diffeomorphism $f \in \diff$ and let $F_1 \oplus \ldots \oplus F_k$
be the  finest dominated splitting over $\La$. Then

\begin{itemize}
\item[a)]
The set $\pfla$ of periodic measures supported in $\La$ is a dense
subset of the set $\mfla$ of invariant measures supported in $\La$.

\item[b)]
For every generic measure $\mu \in \mfla$,

\begin{itemize}
\item[b.i)]
$\mu$ is ergodic;

\item[b.ii)]
$\mu$ has full support: $\supp(\mu) = \La$;

\item[b.iii)]
$\mu$ has zero entropy: $h_\mu(f) = 0$;

\item[b.iv)]
for $\mu$-a.e. point $x$ the Oseledets splitting coincides with $F_1(x) \oplus \ldots \oplus F_k(x)$;

\item[b.v)] $\mu$ is \emph{nonuniformly} hyperbolic.
\end{itemize}

\item[c)]
There exists a dense subset $\cD$ of $\mfla$ such that every $\nu
\in \cD$,

\begin{itemize}
\item[c.i)]
is ergodic;

\item[c.ii)]
has positive entropy: $h_{\nu}(f) > 0$;

\item[c.iii)]
is \emph{uniformly} hyperbolic.

\end{itemize}

\end{itemize}

\end{theorem}
\begin{remark}\begin{enumerate}
\item The conclusion of Theorem \ref{theo4} does not apply to
isolated transitive sets of arbitrary diffeomorphisms: consider for
example a normally hyperbolic irrational rotation of the circle
inside a two-dimensional manifold.
\item Recently D\'\i az and Gorodetski \cite{DG} have shown that
non-hyperbolic homoclinic classes of $C^1$-generic diffeomorphisms
always support at least one  ergodic measure which is not nonuniformly hyperbolic.
\end{enumerate}
\end{remark}

Theorem \ref{theo4}~parts (a)~and~(b) is a nonhyperbolic, $C^1$-generic version of the
following theorem by Sigmund on  hyperbolic
basic sets:

\begin{sigmund}
Given $\La$ a \emph{hyperbolic} isolated transitive set of a
diffeomorphism $f \in \diff$, then the set $\pfla$ of periodic measures in $\La$ is a dense subset of
the set $\mfla$ of invariant measures in $\La$.  Moreover every generic measure $\mu \in \mfla$ is ergodic,
$\supp(\mu) = \La$, and $h_{\mu}(f) = 0$.
\end{sigmund}

\medskip

\begin{remark}
Although this was not stated by Sigmund,  the statement of Theorem \ref{theo4} part (c)
applies also
to the space of measures over any non-trivial hyperbolic basic set.
\end{remark}

\subsection{Approximation by periodic measures}
Many of our results rely in a
fundamental way on the approximation of invariant measures  by periodic measures.
The following theorem
 is at the heart of the proofs of both
Theorem \ref{theo2} and Theorem \ref{theo4}.

\begin{theorem}\label{theo3}
Given an ergodic measure $\mu$ of a $C^1$-generic diffeomorphism $f$,
there is a sequence $\gamma_n$ of periodic orbits such that

\begin{itemize}
\item[i)]
the measures $\mu_{\gamma_n}$ converge to $\mu$ in the weak
topology;

\item[ii)]
the periodic orbits $\gamma_n$ converge to $\supp(\mu)$ in the
Hausdorff topology;

\item[iii)]
the Lyapunov vectors $L(\mu_{\gamma_n})$ converge to the Lyapunov
vector $L(\mu)$.
\end{itemize}
\end{theorem}

As already said, the main novelty here is that,
at the same time, the Lyapunov exponents of the periodic measures converge to those of the measure $\mu$.
Theorem \ref{theo3} is a generic consequence of the perturbative result Proposition~\ref{p.Lyap} which  refines Ma\~n\'e's Ergodic
Closing Lemma.

Consider now the finest dominated splitting supported by the ergodic measure $\mu$. Then \cite{BGV} produces perturbations of the derivative of $Df$ along the orbits of the periodic orbits $\gamma_n$ which make all of the exponents inside a given subbundle coincide.
One deduces:

\begin{coro}\label{cor.bgv}
Given an ergodic measure $\mu$ of a $C^1$-generic diffeomorphism $f$, let  $F_1 \oplus \ldots \oplus F_k$ be the finest dominated splitting over $\supp(\mu)$. Then
there is a sequence of periodic orbits $\gamma_k$ which converges to $\mu$ in the weak topology, to $\supp(\mu)$ in the
Hausdorff topology,  and  such that for each $i \in \{1, \ldots, k\}$ the Lyapunov exponents of
$\gamma_k$ inside $F_i$ converge to the mean value $\la_{E_i}$  of the Lyapunov exponents of $\mu$ inside the $F_i$.
\end{coro}

We state another result which allows to approximate measures
by periodic measures contained in a homoclinic class.

\begin{theo}\label{t.convexsum} For any $C^1$-generic diffeomorphism
$f\in\diff$, any open set $V\subset M$
and any relative homoclinic class $\La=H_V(\cO)$ of $f$ in $V$,
the closure (for the weak topology) of the set $\pfla$ of periodic measures supported
in $\La\cap V$ is convex.
\end{theo}

In other words, every convex sum of periodic measures in $\La\cap V$
is the weak limit of periodic orbits in $\La\cap V$.

\subsection{$C^1$-Pesin theory for dominated splittings}

Theorems \ref{theo1}, \ref{theo2}, and \ref{theo4} constitute as a
group an assault on Obstacle 1. Our next result deals with Obstacle
2. Pugh has built a $C^1$-diffeomorphism which is a counter-example~\cite{Pu2}
to Pesin's Stable Manifold Theorem.
It turns out however that Pesin's Stable Manifold Theorem
\emph{does} hold for maps which are \emph{only} $C^1$, as long as
the $C^1$+H\"{o}lder hypothesis is replaced by a uniform domination
hypothesis on the measure's Oseledets splitting. This has been already
done by Pliss~\cite{Pl}
in the case when all the exponents are strictly negative. The difficulty for applying Pliss argument
when the measure has positive and negative exponents is that we have no control on the geometry of iterated
disks tangent to the stable/unstable directions. The dominated splitting provides us this control
solving this difficulty.

Theorem~\ref{theo5} below is a simpler statement of our complete result
stated in Section~\ref{pesinsection}, where we show that Pesin's Stable Manifold Theorem
applies to ergodic nonuniformly hyperbolic measures with dominated hyperbolic Oseledets
splitting .

\begin{theorem} \label{theo5}
Let $\mu \in \mfm$ be an ergodic nonuniformly hyperbolic measure of
a diffeomorphism $f \in \diff$. Assume that its hyperbolic Oseledets
splitting $\te^s \oplus \te^u$ is dominated.

 Then, for $\mu$-a.e. $x$, there is $\varepsilon(x)>0$ such that
the local stable set $W^s_{\varepsilon(x)}(x)$ is an embedded $C^1$ disc, tangent at $x$ to $\te^s(x)$ and contained in the stable set $W^s(x)$. Furthermore, one can choose $\varepsilon(x)$ in such a way that $x\mapsto \varepsilon(x)$ is a measurable map and such that the family $W^s_{\varepsilon(x)}(x)$ is a measurable family of discs.
\end{theorem}
\noindent In other words:

\emph{(Nonuniform hyperbolicity) + (Uniform domination) $\Rightarrow$
(Pesin Stable Manifold Theorem).}

\noindent We note that its statement includes no genericity assumption on
$f$ or on $\mu$. It has already been used, in a $C^1$-generic context, to
obtain results on ergodic measures of diffeomorphisms far from
homoclinic tangencies \cite{J}.

Theorem \ref{theo5} seems to be a folklore result.
Indeed, R. Ma\~n\'e
\cite{M3} announced this result without proof\footnote{He did
provide the following one-line proof: ``This follows from the
results of Hirsch, Pugh, and Shub.'' Since the ingredients for the
proof we provide in Section \ref{pesinsection} are all classical and
were available in $1982$, we believe that Ma\~n\'e did indeed know
how to prove it, but never wrote the proof (possibly because at the
time there was little motivation for obtaining a Pesin theory for
maps which are $C^1$ but \emph{not} $C^1$+H\"{o}lder).} in his ICM
address.
Although no one seems to have written a full proof under our very
general hypotheses, some authors have used the conclusion implicitly in their work.
Gan \cite{G}, for instance,
uses this kind of idea to extend Katok's
celebrated result on entropy and horseshoes of $C^{1 +
\alpha}$ surface diffeomorphisms to a $C^1$-diffeomorphisms.

\medskip

Theorems \ref{theo2} and \ref{theo4} show that dominated hyperbolic
Oseledets splittings occur quite
naturally in the $C^1$-generic context. We thus obtain:

\begin{corollary} \label{coro1}
Let $f$ be a $C^1$-generic diffeomorphism.
Then for any generic ergodic measure $\mu$,
$\mu$-a.e. $x$ exhibits a $C^1$ stable local manifold $W_{loc}^s(x)$ tangent
to $E(x)$ at $x$ as in Theorem~\ref{theo5}.
\end{corollary}

\begin{corollary} \label{coro2}
Let $\La$ be an isolated transitive set of a $C^1$-generic diffeomorphism $f$. Then for any generic ergodic measure $\mu$,
$\mu$-a.e. $x$ exhibits a $C^1$ stable local manifold $W_{loc}^s(x)$ tangent
to $E(x)$ at $x$ as in Theorem~\ref{theo5}.
\end{corollary}

\subsection{Genericity of irregular points}

Our final two results make precise  some informal statements of
Ma\~n\'e\footnote{``In general, regular points are few from the topological point
of view -- they form a set of first category''. \cite[Page 264]{M4}} regarding the irregularity of generic
points of $C^1$-generic diffeomorphisms.

\begin{theo}\label{irregular0} Given any $C^1$-generic diffeomorphism $f\in \diff$ there is a
residual subset $R\subset M$ such that every $x\in R$ is irregular.
\end{theo}

This result does not hold if we replace regular points by regular$^+$ points: every point in the basin of a (periodic) sink is
regular$^+$. We conjecture that if one excludes the basins of sinks, generic points of $C^1$-generic diffeomorphisms are irregular$^+$.
Our next result is that this conjecture is true in the setting of \emph{tame} diffeomorphisms\footnote{Indeed a recent result by J. Yang \cite{Y} allows us to extend Theorem~\ref{t.irregular} to $C^1$-generic diffeomorphisms far from tangencies: Yang announced that, in this setting, generic points belongs to the stable set of homoclinic classes.}.

Recall that a diffeomorphism is called \emph{tame} if all its
chain recurrence classes are robustly isolated (see~\cite{BC}). The set of tame
diffeomorphisms is a $C^1$-open set which strictly contains the
set of Axiom A+no cycle diffeomorphisms. The chain recurrent set
of $C^1$-generic tame diffeomorphisms consist of finitely many
pairwise disjoint homoclinic classes.  Our result is :

\begin{theo}\label{t.irregular}
If $f$ is a $C^1$-generic tame diffeomorphism then there is a
residual subset $R\subset M$ such that if $x\in M$ and $\omega(x)$
is not a sink, then $x$ is both irregular$^+$ and Lyapunov-
irregular$^+$.
\end{theo}


\subsection{Layout of the Paper}
The remainder of this paper is organized as follows:

\begin{itemize}

\item
In Section \ref{mainlemmasection} we prove
an ergodic analogue of Pugh's General Density Theorem which we call
\emph{Ma\~n\'e's Ergodic General Density Theorem}. It implies
items (i) and (ii) of Theorem \ref{theo3}.
We also prove a ``generalized specification property" satisfied by $C^1$-generic
diffeomorphisms inside homoclinic classes: this gives Theorem~\ref{t.convexsum}.
One deduces from these results the parts (a) and (c) of Theorem \ref{theo4}.

\item
In Section \ref{theo4section} we state and prove some abstract
results on ergodicity, support, and entropy of generic measures. We
show then how these abstract results yield Theorem~\ref{theo1},
item (i) of Theorem~\ref{theo2} and items (b.i), (b.ii), and
(b.iii) of Theorem \ref{theo4}.

\item
In Section \ref{lyapexponents} we control the Lyapunov exponents of the periodic measures
provided by Ma\~ n\'e's ergodic closing lemma. This implies the item (iii) of Theorem
\ref{theo3}.

\item
In Section \ref{prelimhyp} we prove Corollary~\ref{cor.bgv}
and we combine most of the previous machinery with some new ingredients in order to obtain our results
on nonuniform hyperbolicity of generic measures:
item (ii) of Theorem~\ref{theo2} and items (b.iv) and (b.v) of Theorem \ref{theo4}.

\item
In Section \ref{pesinsection} we construct an adapted
metric for the Oseledets splittings and then use it
to prove Theorem \ref{theo5}.

\item
Finally, in Section \ref{irregular} we prove Theorems \ref{irregular0} and \ref{t.irregular}.
\end{itemize}


\paragraph{Acknowledgements.}
 The authors would like to thank the following people for useful suggestions and comments:
A. Avila, A. Baraviera, J. Bochi, F. B\'eguin, L. J. Diaz, F. Le Roux, L.Wen, J. Yang.

This paper has been prepared during visits of the authors at the Universit\'e de Paris XIII, at the Universit\'e de Bourgogne and at IMPA, which were financed by the Brazil-France Cooperation Agreement on Mathematics. The text has been finished during the Workshop on Dynamical Systems
at  the ICTP in july 2008.
We thank these institutions for their kind hospitality.


\section{Approximation of invariant measures by periodic orbits} \label{mainlemmasection}

\subsection{Ma\~n\'e's Ergodic General Density Theorem}

In \cite{M4} Ma\~n\'e states without proof the following fact (called Ma\~n\'e's Ergodic General Density Theorem):

\begin{theo}\label{M'sEGDT}For any $C^1$-generic diffeomorphism $f$,  the  convex hull of periodic measures
is  dense in $\mfm$.

More precisely, every  measure $\mu\in\mfm$ is approached in the weak topology by a measure $\nu$ which is the convex sum of finitely many periodic measures and whose support $\supp(\nu)$ is arbitrarily close to $\supp(\mu)$.
\end{theo}

We now prove a more precise result which corresponds to items (i) and (ii) of Theorem \ref{theo3}:  the
ergodic measures are approached by periodic measures in the
weak and Hausdorff senses.
In Section \ref{lyapexponents} we shall modify the
proof in order to include also the approximation of the mean
Lyapunov exponents in each bundle of the finest dominated
splitting (item (iii)).

\begin{theo}\label{ergodlemma}
Given $\mu \in \mfm$ an ergodic measure of a $C^1$-generic diffeomorphism
$f$, then for every neighborhood $\cV$ of $\mu$ in $\mfm$ and every
neighborhood $\cW$ of $\supp(\mu)$ in $\cK(M)$ there is some
periodic measure $\mu_\gamma$ of $f$ such that $\mu_\gamma \in \cV$ and
$\supp(\mu_\gamma) \in \cW$.
\end{theo}

Just as the $C^1$-generic density of  $Per_f(M)$  in
$\Omega(f)$ follows from Pugh's Closing
Lemma \cite{Pu1}, Theorem~\ref{ergodlemma} follows
from Ma\~n\'e's Ergodic Closing Lemma \cite{M2} (discussed below).

\begin{definition}
A (recurrent) point $x$ of $f \in \diff$ is \emph{well-closable} if
given any $\eps > 0$ and any neighborhood $\cU$ of $f$ in $\diff$
there is some $g \in \cU$ such that $x \in Per_g(M)$ and moreover
$$d(f^k(x), g^k(x)) < \eps$$
\noindent for all $k\geq 0$ smaller than the period of $x$ by $g$.
\end{definition}

That is, a point $x$ is well-closable if its orbit can be closed via
a small $C^1$-perturbation in such a way that the resulting periodic
point ``shadows'' the original orbit along the periodic point's
entire orbit. Ma\~n\'e proved that almost every point of any invariant measure is
well-closable:

\begin{ergodicclos}
Given $f \in \diff$ and $\mu \in \mfm$, $\mu$-a.e. $x$ is
well-closable.
\end{ergodicclos}

Birkhoff's ergodic theorem and Ma\~n\'e's ergodic closing lemma implies:

\begin{coro}\label{ergodicclosmet}
Given $f \in \diff$ and an ergodic measure $\mu \in \mfm$, then for
any neighborhoods $\cU$ of $f$ in $\diff$ and $\cW$ of $\mu$  in
$\mm$ and  any $\varepsilon>0$  there is $g \in \cU$ having a periodic orbit
$\gamma$  such that $\mu_\gamma \in \cW$ and the Hausdorff distance between $\supp(\mu)$ and $\gamma$
is less than $\varepsilon$.
\end{coro}

\begin{proof}[Proof of Theorem~\ref{ergodlemma}]

We consider  $X  :=\mm \times \cK(M)$ endowed  with the product metric. The
space $\cK(\mm \times \cK(M))$ of compact subsets of $\mm \times
\cK(M)$ is a compact metric space when endowed with the Hausdorff
distance. Consider the map $\Phi: \diff \longrightarrow \cK(\mm
\times \cK(M))$, which associates to each diffeomorphism $f$  the
closure of the set of pairs $(\mu_\gamma,\gamma)$ where $\gamma$
is a periodic orbit of $f$.

Kupka-Smale Theorem asserts that there is a residual set $\cR$ of
$\diff$ such that every periodic orbit of $g\in\cR$ is
hyperbolic. Then the robustness of hyperbolic
periodic orbits implies that the map $\Phi$ is
lower-semicontinuous at $g \in \cR$. Applying the Semicontinuity
Lemma (see Section~\ref{ss.semi}) to $\Phi|_{\cR}$,
we obtain a residual subset $\cS$ of $\cR$
(and hence of $\diff$) such that every $g \in \cS$ is a continuity
point of $\Phi|_{\cR}$. We shall now show that each such continuity
point satisfies the conclusion of Lemma \ref{ergodlemma}:

Consider $g\in\cS$  and $\mu$  an
ergodic measure of $g$. Fix an open
neighborhood $\cZ_0$ of $(\mu, \supp(\mu))$ in $\cK(\mm \times
\cK(M))$. We need to prove  that there exists a pair
$(\mu_{\gamma}, \gamma)$ in ${\cZ_0}$, where $\gamma$ is a periodic orbit of $g$.
Fix now a compact neighborhood $\cZ\subset \cZ_0$ of  $(\mu, \supp(\mu))$;
it is enough to prove that $\Phi(g)\cap \cZ\neq\emptyset$.

Applying the Corollary~\ref{ergodicclosmet} to $g$, we obtain  an
arbitrarily small $C^1$-perturbation $g'$ of $g$ having a periodic orbit
$\gamma$  such that simultaneously $\mu_{\gamma}$ is
weak-close to $\mu$ and $\gamma$ is Hausdorff-close to
$\supp(\mu)$.
With another arbitrarily small $C^1$-perturbation
$g''$ we make $\gamma$ hyperbolic and hence robust, while keeping
$\mu_{\gamma}$ close to $\mu$ and $\gamma$ close to $\supp(\mu)$.
With yet another small $C^1$-perturbation $g'''$, using the
robustness of $\gamma$, we guarantee that $g''' \in \cR$ and $(\mu_\gamma,\gamma)\in \cZ\cap\Phi(g''')$.

By letting $g'''$ tends to $g$, using the continuity of  $\Phi|_\cR$ at $g$ and the compactnes of $\cZ$
one gets that $\cZ\cap \Phi(g)\neq \emptyset$ as announced.
\end{proof}

Theorem~\ref{M'sEGDT} now follows by combining Theorem \ref{ergodlemma} with the
 the following ``approximative'' version of the Ergodic
Decomposition Theorem, which is easily deduced from the standard
statement:

\begin{ergodicdec}
Given a homeomorphisms $f $ of a compact metric space $M$ and $\mu \in \mfm$, then for any neighborhood
$\cV$ of $\mu$ in $\mfm$ there are a finite set of ergodic measures
$\mu_1, \ldots, \mu_k \in \mfm$ and positive numbers $\la_1, \ldots,
\la_k$ with $\la_1 + \ldots + \la_k = 1$ such that
$$\la_1 \mu_1 + \ldots + \la_k \mu_k \in \cV.$$
\end{ergodicdec}

That is, any invariant measure may be approached by finite
combinations of its ergodic components.

\begin{proof}[Proof of Theorem~\ref{M'sEGDT}]

Let $\mu$ be an invariant measure of a $C^1$-generic diffeomorphism $f$ as
in the statement of Lemma \ref{ergodlemma}.  Fix a
neighborhood $\cV$  of $\mu$ in $\mfm$ and  a number $\varepsilon>0$.

Let $\tilde\cV$ denote $\cV\cap\cM_f(\supp(\mu))$: it is a neighborhood of $\mu$ in $\cM_f(\supp(\mu))$. By the Ergodic Decomposition
Theorem applied to $f|_{\supp(\mu)}$ and $\tilde\cV$ there is a convex combination
$$\la_1 \mu_1 + \ldots + \la_k \mu_k$$

\noindent of ergodic measures which belongs to $\cV$ and supported in $\supp(\mu)$.
Now, by Theorem~\ref{ergodlemma}, each ergodic component $\mu_i$ is weak-approached
by periodic measures $\mu_{\gamma_i}$ of $f$ whose support $\gamma_i$ is contained in the $\varepsilon$-neighborhood of $\supp(\mu_i)$ and hence of $\supp(\mu)$.

Now the convex sum $\nu = \la_1 \mu_{\gamma_1} + \ldots + \la_k \mu_{\gamma_k}$ is close to $\mu$ for the weak topology and its support is contained in the $\varepsilon$-neigborhood of $\supp(\mu)$. As the support of a measure varies lower-semicontinuously with the measure in the weak topology, we get that $\supp(\mu)$ and $\supp(\nu)$ are close in the Hausdorff distance.

\end{proof}

\subsection{Periodic measures in homoclinic classes of $C^1$-generic diffeomorphisms}
\label{mlemmaingredients}

Through the use of Markov partitions, Bowen \cite{Bow}
showed that every hyperbolic basic set $\Lambda$ contains
periodic orbits with an arbitrarily prescribed itinerary
(this is known as the \emph{specification property}).
So the invariant probabilities supported in $\Lambda$
are approached in the weak topology by periodic orbits
 in $\Lambda$. An intermediary step for this
result consists in proving
that every convex sum of periodic measures in $\Lambda$
is approached by periodic orbits in $\Lambda$.
One thus defines:
\begin{definition}\label{transition}
A set of periodic points $\Sigma \subset Per_f(M)$
has the \emph{barycenter property} if, for any two points $p ,q\in  \Sigma$, any  $\la\in[0,1]$, and $\eps > 0$, there exists $x \in
\Sigma$ and pairwise disjoint sets $I, J\subset \NN\cap [0,\Pi(x))$ such that
\begin{enumerate}
\item
$\la - \eps < \frac{Card(I)}{\Pi(p)} <
\la + \eps$  and  $(1-\la) - \eps < \frac{Card(J)}{\Pi(q)} <
(1-\la) + \eps$,

\item
$d(f^m(x), f^m(p)) < \eps$ for every $ m \in I$ and
$d(f^m(x), f^m(q)) < \eps$ for every $m \in J$.
\end{enumerate}
\end{definition}

The barycenter property implies that for any two
periodic points $p,q$ in $\Sigma$ and $\lambda\in(0,1)$
 there is some periodic point $x \in \Sigma$, of very high
period, which  spends a portion
approximately equal to $\la$ of its period shadowing the orbit of
$p$ and a portion equal to $1-\la$ shadowing the orbit of $q$. As a consequence we get:

\begin{remark}\label{r.trans} If a set $\Sigma\subset\pfm$ has the barycenter property then the closure of $\{m_{\cO(p)}, p\in\Sigma\}\subset \mfm$ is convex.
\end{remark}

Consider now the set $\Sigma_p$ of periodic
points homoclinically related to a hyperbolic periodic point $p$ of an arbitrary diffeomorphism $f$.
Then $\Sigma_p$ is contained in an increasing
sequence of basic sets contained in the homoclinic class $H(p)$. For this reason, it remains
true that every convex sum of periodic measure $\mu_{\gamma_i}$ with
$\gamma_i\in \Sigma(p)$ is approached by a periodic orbit in the basic set.
From the transition property in \cite{BDP}, we thus have:

\begin{proposition} \label{trans}
For any open set $V\subset M$ and any hyperbolic periodic point
$p$ whose orbit is contained in $V$, the set $\Sigma_{V,p}$ of periodic points related to $p$ in $V$
satisfies the barycenter property.
\end{proposition}

Proposition~\ref{trans} does not hold a priori for the set of periodic orbits
in an homoclinic class $H(\cO)$ in particular in the case where
$H(\cO)$ contains periodic points of different indices
(which thus are not homoclinically related).
However when two hyperbolic periodic orbits $\gamma_1,\gamma_2$ of different
indices are related by a heterodimensional cycle, \cite{ABCDW} shows that
one can produce, by arbitrarily small $C^1$-perturbations, periodic orbits
which spend a prescribed proportion of time shadowing the orbit of $\gamma_i$, $i=1,2$.
Furthermore, if $\gamma_1$ and $\gamma_2$ are robustly in the same chain recurrence class,
then the new orbits also belongs to the same class.
This allows one to prove that the barycenter property holds generically:

\begin{proposition} \label{trans2} Let $f$ be a $C^1$-generic diffeomorphism and
$\cO$ a hyperbolic periodic orbit and $V\subset M$ be an open set.
Then the set  of periodic orbits contained in $V\cap H_V(\cO)$
satisfies the barycenter property.
\end{proposition}
Notice that this proposition together with Remark~\ref{r.trans}
implies Theorem~\ref{t.convexsum}.

\begin{proof}
We first give the proof for whole homoclinic classes (i.e. when $V=M$).
According to \cite{BD2}, for every $C^1$-generic diffeomorphism $f$ and every periodic point $p,q$ of $f$ the homoclinic classes are either equal or disjoint; furthermore, if $H(p)=H(q)$ then there is an open neighborhood $\cU$ of $f$ such that for every generic $g\in \cU$ the homoclinic classes of the continuations of $p$ and $q$ for $g$ are equal; moreover if  $H(p)=H(q)$ and if $p$ and $q$ have the same index, then they are homoclinically related. Hence the barycenter property is satisfied for pairs of point of the same index in an homoclinic class.

Hence we now assume that $H(\cO)$ contains periodic points $p$ and $q$ with different
indices, and we fix some number $\lambda\in (0,1)$.
We want to prove the barycenter property for $p$, $q$ and $\lambda$.
Notice that the homoclinic classes of $p$ and $q$ are not trivial
and from~\cite{BC}, one may assume that they coincide with $H(\cO)$.
The next lemma will allow us to assume that $p$ and $q$ have all their eigenvalues real,
of different modulus, and of multiplicity equal to $1$.

\begin{lemma} Let $f$ be a $C^1$-generic diffeomorphism and $p$ be
a periodic point of $f$ whose homoclinic class is non-trivial.
Then for every $t\in(0,1)$ and $\varepsilon>0$ there is a periodic point
$p_\varepsilon$ homoclinically related with $p$,
and a segment $I=\{i,i+1,\dots,i+j\}\subset [0,\Pi(p_\varepsilon)-1]\cap \NN$  with
$1-\frac j{\Pi(p)}<\varepsilon$ such that:
\begin{itemize}
\item  for every $k\in\{0,\dots,j\}$ one has $d(f^{i+k}(p_\varepsilon),f^k(p))<\varepsilon$,
\item  the eigenvalues of $Df^{\Pi(p_\varepsilon)}(p_\varepsilon)$ are real;
have different modulus, and multiplicity equal to $1$;
\item $p$ and $p_\varepsilon$ have the same index of $p$.
\end{itemize}
\end{lemma}

\begin{proof} The proof consists in considering periodic orbits of very large period
shadowing the  orbit of an homoclinic intersection associated to $p$.
An arbitrarily small perturbation of the derivative of such orbits produces eigenvalues
that are real, have different modulus and multiplicity $1$.
As this property is an open property, the genericity assumption implies that
$f$ already exhibits the announced periodic orbits, without needing perturbations.
\end{proof}
Notice that if, for every $\varepsilon>0$, the barycenter property is satisfied for
$p_\varepsilon$, $q$ and $\lambda$, then it also holds for  $p$, $q$ and $\lambda$.
Hence we may assume that the points $p$ and $q$ have different indices and
have all their eigenvalues real, of different modulus, and of multiplicity equal to $1$.
For fixing the idea one assume $\dim W^s(p)<\dim W^s(q)$.
Furthermore $H(p)=H(q)$ from~\cite{BD2}
and this property persists for any $C^1$-generic diffeomorphism close to $f$.

The end of the proof now follows from \cite{ABCDW};
however there is no precise statement in this paper of the result we need.
For this reason we recall here the steps of the proof.
First  by using Hayashi connecting lemma, one creates an heterodimensional cycle
associated to the points $p$ and $q$: one has
$W^u(\cO(p))\trans W^s(\cO(q))\neq\emptyset$
and $W^s(\cO(p))\cap W^u(\cO(q))\neq \emptyset$.
Then \cite[Lemma 3.4]{ABCDW} linearizes the heterodimensional cycle producing
an \emph{affine heterodimensional cycle}.
This heterodimensional cycle \cite[Section 3.2]{ABCDW}
produces, for every large $\ell,m$, a periodic point $r_{\ell,m}$
whose orbit spends exactly $\ell.\Pi(p)$ times shadowing the orbit of $p$
and $m.\Pi(q)$ times shadowing the orbit of $q$
and an bounded time outside a small neighborhood of these two orbits.
So, we can choose $\ell$ and $m$ such that the orbit of $r_{\ell,m}$
spends a proportion of time close to the orbit of $p$ which is almost $\lambda$
and a proportion of time close to the orbit of $q$ which is almost $1-\lambda$.
Furthermore, one has  $W^u(\cO(p))\trans W^s(r_{\ell,m})\neq\emptyset$ and-
$W^u(r_{\ell,m})\trans W^s(\cO(q))\neq\emptyset$.
One deduces that, for any $C^1$-generic diffeomorphism in an open set close to $f$,
$H(r_{\ell,m})=H(p)=H(q)=H(\cO)$.
Since $f$ is generic, the class $H(\cO)$ for $f$ already contained periodic orbits
that satisfy the barycenter property.
\bigskip

In the proof for relative homoclinic classes,
there are several new difficulties: the relative homoclinic class of $\cO$
in an open set $V$  is the closure of periodic orbits in $V$ related to $\cO$
by orbits in $V$, but some periodic orbit may also be contained in the closure of $V$.
Furthermore the set of open sets is not countable:
hence the set of relative homoclinic classes is not countable, leading to some difficulty
for performing an argument of genericity. We solved these difficulties by considering
the set $Per_f(H_V(\cO)\cap V)$ of  periodic orbits of $H_V(\cO)$
which are contained in $V$.
Then, if $V$ is an increasing union of open subsets $\cdots \subset V_n\subset
V_{n+1}\subset \cdots$ then
$$Per_f(H_V(\cO)\cap V)=\bigcup_n Per_f(H_{V_n}(\cO)\cap V_n).$$
This argument allows us to deal with a countable family of open sets $V_i$, $i\in\NN$.
One now argues in a very similar way as before
(just taking care that all the orbits we use are contained in the open set $V$).
\end{proof}

\subsection{Approximation of measures in isolated transitive sets}
One of the main remaining open question for $C^1$-generic diffeomorphisms is
\begin{question} Given a $C^1$-generic diffeomorphism $f\in\diff$ and a homoclinic class
$H(p)$ of $f$, is $\cP_f(H(p))$ dense in $\mfhp$? In other words, is every measure supported on $H(p)$ approached by periodic orbits inside the class?
\end{question}
The fact that we are not able to answer to this question is the main reason for which we will restrict the study to isolated transitive set classes, in this section.

An argument by Bonatti-Diaz \cite{BD2}, based on Hayashi
Connecting Lemma, shows that isolated transitive sets $\La$
of $C^1$-generic diffeomorphisms are  relative homoclinic classes:

\begin{theorem}\cite{BD2}\label{bd2}
Given $\La$ an isolated transitive set of a $C^1$-generic diffeomorphism
$f$ and let $V$ be an isolating open neighborhood of $\La$, then
$$\La = H_V(\cO)$$
\noindent for some periodic orbit $\cO\subset \La$ of $f$.
\end{theorem}

\begin{proof}[Proof of Theorem~\ref{theo4} part (a)]
Let $\La$ be an isolated transitive set of a $C^1$-generic diffeomorphism
$f$ and $\mu$ be an invariant measure supported on $\La$. According to Theorem~\ref{M'sEGDT},
the  measure $\mu\in\mfm$ is approached in the weak topology by a measure
$\nu$ which is the convex sum of finitely many periodic measures and
whose support $\supp(\nu)$ is arbitrarily close to $\supp(\mu)$.

On the other hand $\La$ is the relative homoclinic class $H_V(p)$ of some periodic point $p \in
\La$ in some isolating open neighborhood $V$; as the support of $\nu$ is close to the support of $\mu$
one gets that $\supp(\nu)$ is contained in $V$.  As $V$ is an isolated neighborhood of $\La$
the measure $\nu$ is in fact supported in $\La$: hence it is the
convex sum of finitely many periodic measures in $\pfla$.

As $\La$ is compact and contained in $V$ it does not
contain periodic orbits on the boundary of $V$. Hence
Theorem~\ref{t.convexsum} implies
that the closure of the set $\pfla$ is convex;
this implies that $\nu$ belongs to the closure of $\pfla$,
ending the proof.
\end{proof}

\begin{proof}[Proof of Theorem \ref{theo4} part (c)]
Let $\La$ be a (non-trivial) isolated transitive set of a $C^1$-generic
diffeomorphism $f$. By Theorem \ref{bd2}, every periodic point $p$
in $\La$ has homoclinic class equal to $\La$, and hence exhibits
some transverse homoclinic orbit. This implies that there are
hyperbolic horseshoes $\Gamma$ arbitrarily close to this homoclinic orbit.
The points in $\Gamma$ spend arbitrarily large
fractions of their orbits shadowing the orbit $\cO(p)$ of $p$ as
closely as we want.

Every horseshoe $\Gamma$ supports  ergodic measures $\nu$
which have positive entropy. Since each such $\nu$ is
supported in a hyperbolic horseshoe, it follows that $\nu$ is also
uniformly hyperbolic. Now, because the periodic horseshoe $\Gamma$
shadows $\cO(p)$ along most of its orbit, it follows that $\nu$
 is
close in the weak topology to the periodic measure $\mu_{\cO(p)}$
associated to the orbit of $p$.

Since by the Theorem~\ref{theo4} part (a) the set of periodic measures $\pfla$ is
dense in $\mfla$, then it follows that the set of ergodic,
positive-entropy, and uniformly hyperbolic measures $\nu$ as above
is also dense in $\mfla$.
\end{proof}

\section{Ergodicity, Support, Entropy}\label{theo4section}

In this section we prove three ``abstract'' results on generic
measures, dealing respectively with their ergodicity, support, and
entropy. These results, together with the Theorem \ref{theo4} part
(a), respectively imply items (b.i), (b.ii), and (b.iii) of
Theorem \ref{theo4}. We also use these general results
to obtain Theorem \ref{theo1} and item (i) of Theorem \ref{theo2}.

\medskip

\subsection{Ergodicity}\label{b.isection}

Let $\La$ be an isolated transitive set of a $C^1$-generic
diffeomorphism $f$. By Theorem \ref{theo4} (a), ergodicity is
a dense property in $\mfla$, since periodic measures are ergodic.

Since dense $G_\delta$ sets are residual, we need only prove that
ergodicity is $G_\delta$ in the weak topology in order to conclude
that ergodicity is generic in $\mfla$. And indeed we have the
following general result (which implies in particular item (b.i) of
Theorem \ref{theo4}):

\begin{proposition} \label{propb.i}
Let $X$ be a compact metric space, $A: X \rightarrow X$ be a
continuous map, and $\mathcal{M}_A(X)$ denote the space of
$A$-invariant Borel probabilities on $X$, endowed with the weak
topology. Then ergodicity is a $G_{\delta}$ property in
$\mathcal{M}_A(X)$. In particular, if there exists a dense subset
$\cD$ of $\mathcal{M}_A(X)$ which consists of ergodic measures, then
every generic measure in $\mathcal{M}_A(X)$ is ergodic.
\end{proposition}

\begin{proof}
Let $\psi \in C^0(X)$ be a continuous real-valued function on $X$.
The set
$$\mathcal{M}_{A, \psi}^{erg}(X) := \left\{{\mu \in \mathcal{M}_A(X): \int \psi \; d{\mu} = \lim_{k \to +\infty} \frac{1}{k} \; \sum_{j =1}^{k} \psi(A^j(x)) \text{ for $\mu$-a.e. $x$}}\right\}$$
\noindent of measures which are ``ergodic with respect to $\psi$''
is given by

$$\mathcal{M}_{A, \psi}^{erg}(X) = \bigcap_{\ell \in \NN} \; \bigcup_{n \in \NN} \;
\left\{\mu \in \mx: \int\|\frac{1}{n} \sum_{j =1}^{n} \psi(f
A^j(x)) - \int \psi \; d{\mu}\| \; d{\mu(x)} <
\frac{1}{\ell}\right\}.$$
In particular $\mathcal{M}_{A,
\psi}^{erg}(X)$ is a $G_\delta$ set: the integral in the
right-hand side of the bracket varies continuously with the
measure $\mu$, and so the set defined within the brackets is open
in $\mx$; this shows that $\mathcal{M}_{A, \psi}^{erg}(X)$ is a
countable intersection of open sets.

Now let $\{\psi_k\}_{k \in \NN}$ be a countable dense subset of
$C^0(X)$. By the argument above, for each $k \in \NN$ there is
some $G_{\delta}$ subset $\cS_k$ of $\mx$ consisting of measures
which are ergodic with respect to $\psi_k$. The measures $\mu$
which belong to the residual subset $\cS$ of $\mx$ obtained by
intersecting the $\cS_k$'s are precisely the measures which are
simultaneously ergodic with respect to every $\psi_k$. Using
standard approximation arguments one can show that such $\mu$ are
ergodic with respect to \emph{any} $\psi \in C^0(X)$, and hence  are
ergodic.
\end{proof}

\begin{remark}
Proposition \ref{propb.i} implies in particular that the space
$\mflerg$ of ergodic measures of a diffeomorphism $f$ is a Baire
space when endowed with the weak topology. Indeed, any $G_\delta$
subset $\cA$ of a compact metric space is Baire, since $\cA$ is
residual in $\over{\cA}$.
\end{remark}

\subsection{Full Support}\label{b.iisection}

Given $\La$ an isolated transitive set of a $C^1$-generic
diffeomorphism $f$, then $\La$ is a homoclinic class, and hence
has a dense subset $\perfla$ of periodic points. This last fact
suffices to prove that generic measures on $\La$ have full support
(item (b.ii) of Theorem \ref{theo4}), as the following general
result shows:

\begin{proposition}\label{propb.ii}
Let $X$ be a compact metric space, $A: X \rightarrow X$ be a
continuous map, and $\mathcal{M}_A(X)$ denote the space of
$A$-invariant Borel probabilities on $X$, endowed with the weak
topology. Then every generic measure $\mu$ in $\mathcal{M}_A(X)$
satisfies
$$\supp(\mu) = \bigcup_{\nu \in \mathcal{M}_A(X)} \; \supp(\nu).$$
In particular, if the set of periodic points of $A$ is dense in $X$,
then every generic $\mu$ satisfies $\supp(\mu) = X$.
\end{proposition}

\begin{proof}
Consider the map
$$\Phi: \mx \rightarrow \cK(X)$$
$\hspace{200pt} \mu \; \; \; \mapsto \supp(\mu).$

It is easy to see that $\Phi$ is lower-semicontinuous. By the
Semicontinuity Lemma, there is a residual subset $\cS$ of $\mx$
which consists of continuity points of $\Phi$. The following claim
then concludes the proof:

\begin{claim*}
Given $\mu \in \mx$ a continuity point of $\Phi$, then $\supp(\mu) =
\bigcup_{\nu \in \mathcal{M}_A(X)} \; \supp(\nu)$.
\end{claim*}

Let us now prove the claim. One considers any measure $\nu\in
\mathcal{M}_A(X)$. The measures $(1 - \la) \mu + \la \, \nu$
converge to $\mu$ as $\la$ goes to zero, and hence their supports,
which equal $\supp(\mu) \cup \supp(\nu)$, converge to
$\supp(\mu)$. This implies that $\supp(\nu)$ is contained in
$\supp(\mu)$ and concludes the proof of the claim.
\end{proof}

\subsection{Zero Entropy}\label{b.iiisection}

The next abstract result shall allow
us to prove item (b.iii) of Theorem \ref{theo4}, that is, that
generic measures of an isolated transitive set have zero entropy:

\begin{proposition}\label{propb.iii}
Let $X$ be a compact metric space, $A: X \rightarrow X$ be a
continuous map, and $\mathcal{M}_A(X)$ denote the space of
$A$-invariant Borel probabilities on $X$, endowed with the weak
topology. Assume that there exists a sequence of measurable finite
partitions $\{\cP_k\}_{k \in \NN}$ of $X$ such that
\begin{itemize}

\item[1)] the partition $\cP_{k + 1}$ is finer than $\cP_k$ for
every $k \in \NN$;

\item[2)]
the product $\bigvee_{k \in \NN} \; \cP_k$ is the Borel
$\sigma$-algebra of $X$.
\end{itemize}

Assume also that there is a dense subset $\cD$ of $\cM_A(X)$ such that every
$\mu \in \cD$ satisfies $\mu(\partial \cP_k) = 0$ and $h(\mu, \cP_k)
= 0$, for every $k \in \NN$. Then there is a residual subset $\cS$
of $\cM_A(X)$ such that every $\mu \in \cS$ satisfies $h(\mu) = 0$.

\end{proposition}

\begin{proof}
By the Kolmogorov-Sinai theorem given any $\mu \in \mx$ then the
entropy $h(\mu)$ of $\mu$ is equal to $sup\;_{k \in \NN} \; \{h(\mu,
\cP_k')\}$.

By assumption, given $k \in \NN$ and $\nu \in \cD$ then
$\nu(\partial \cP_k) = 0$. Thus $\nu$ is a point of
upper-semicontinuity for the map
$$\Theta_k: \mx \rightarrow \RR$$
$\hspace{230pt} \mu \mapsto h(\mu, \cP_k).$

Since $\Theta_k(\nu) = 0$ at every $\nu \in \cD$, it
follows that every $\nu$ is in fact a continuity point of
$\Theta_k$. Since $\cD$ is dense in $\cM_A(X)$, these conditions imply that there
exists a residual subset $\cS_k$ of $\mx$ consisting of measures
$\mu$ such that $h(\mu, \cP_k) = 0$.

Setting $\cS := \bigcap_{k \in \NN} \cS_k$ we obtain a residual
subset of $\mx$ which consists of measures $\mu$ such that $h(\mu,
\cP_k) = 0$ for every $k \in \NN$, and hence which by the
Kolmogorov-Sinai theorem satisfy $h(\mu) = 0$.
\end{proof}

We may now prove item (b.iii) of Theorem \ref{theo4}:

\begin{corollary}
Given $\La$ an isolated transitive set of a $C^1$-generic
diffeomorphism $f$, then there is a residual subset $\cS$ of
$\mfla$ such that every $\mu \in \cS$ has zero entropy.
\end{corollary}

\begin{proof}
From Theorem~\ref{theo4}, part (a), there is a dense
subset $\cD$ of $\mfla$ which consists of periodic measures.

Let now $\{\cP_k\}_{k \in \NN}$ be a sequence of finite partitions
of $M$ into zero-codimension submanifolds of $M$ and their
boundaries such that:

\begin{itemize}

\item[1)]
$\partial \cP_k \cap \perfla = \empty$ for all $k \in \NN$;

\item[2)]
the partition $\cP_{k + 1}$ is finer than $\cP_k$ for every $k \in
\NN$;

\item[3)]
the product $\bigvee_{k \in \NN} \; \cP_k$ is the Borel
$\sigma$-algebra of $M$.
\end{itemize}

Then the intersection of each of the partitions $\cP_k$ with $\La$
yields a sequence of partitions $\{\cP_k'\}_{k \in \NN}$ of $\La$
which satisfy conditions (1)-(3) above (replacing $M$ by $\La$ in
condition (3)).

Clearly this sequence of partitions $\{\cP_k'\}_{k \in \NN}$
satisfies the hypotheses of Proposition \ref{propb.iii} above, with
$X = \La$ and $\cD$ the set of periodic measures supported in $\La$.
So there is a residual subset of
$\mfla$ consisting of measures with zero entropy.
\end{proof}

We can also use Proposition \ref{propb.iii} to prove that generic
measures of $C^1$-generic diffeomorphisms have zero entropy.
Indeed, by Theorem~\ref{M'sEGDT}, given a $C^1$-generic diffeomorphism $f$ then
the set of finite convex combinations of periodic measures of $f$
is dense in $\mfla$. Moreover, given a sequence of partition
$\{\cP_k\}_{k \in \NN}$ of $M$ as in the proof of Proposition
\ref{propb.iii}, then each such combination satisfies the
hypotheses of Proposition \ref{propb.iii} above. So we obtain
the item (i) of Theorem~\ref{theo2}.

\begin{corollary}
For any $C^1$-generic diffeomorphism $f$, every generic measure
$\mu$ of $f$ satisfies $h(\mu) = 0$.
\end{corollary}

\subsection{Ergodic measures whose support fills a homoclinic class}\label{theo1section}

Given an open set $V\subset M$
and a hyperbolic periodic point $p$ with orbit contained in $V$,
recall that $\Sigma_{V,p}$
denotes the set of periodic orbits heteroclinically related to (the
orbit of) $p$ by orbits contained in $V$.
Let $\cM(\Sigma_{V,p})$ denote the set of periodic measures
associated to orbits in $\Sigma_{V,p}$ and let $\cC\cM(\Sigma_{V,p})$ denote
the closure (in the weak topology) of the convex hull of
$\cM(\Sigma_{V,p})$.
We may now state:

\begin{theo1'}
Let $H_V(p)$ be a homoclinic class of some $f \in \diff$. Then there
is a residual subset $\cS$ of $\cC\cM(\Sigma_{V,p})$ such that every
$\mu \in \cS$ is ergodic, satisfies $\supp(\mu) = H_V(p)$, and has
zero entropy.
\end{theo1'}

In particular, at least one such measure exists, implying Theorem
\ref{theo1}. So it turns out that the proof of Theorem \ref{theo1}
-- whose statement includes no genericity conditions at all --
ultimately relies on generic arguments on the space of measures
supported in $\Sigma_{V,p}$; this is a good illustration of the capacity
of genericity arguments to yield non-generic results.

\begin{proof}[Proof of Theorem \ref{theo1}']
By Proposition \ref{trans}, the set $\cM(\Sigma_{V,p})$ of periodic
measures associated to orbits in $\Sigma_{V,p}$ constitutes a dense
subset of $\cms$. That is, we have that $\cms =
\over{\cM(\Sigma_{V,p})}$.

Now, each element of $\cM(\Sigma_{V,p})$ is ergodic and so it follows by
Lemma \ref{propb.i} that there is some residual subset $\cS_1$ of
$\cms$ such that every $\mu \in \cS_1$ is ergodic. Lemma
\ref{propb.ii} implies that there is some residual subset $\cS_2$ of
$\cms$ such that the support of every $\mu \in \cS_2$ coincides with
$\over{\Sigma_{V,p}} = H_V(p)$. And by Lemma \ref{propb.iii} there is some
residual subset $\cS_3$ of $\cms$ such that every $\mu \in \cS_3$
has zero entropy. Set $\cS := \cS_1 \cap \cS_2 \cap \cS_3$ and
we are done.
\end{proof}

\section{Approximation of Lyapunov Exponents by Periodic Orbits}\label{lyapexponents}
One deduces Theorem~\ref{theo3} from the following perturbative result:

\begin{prop}\label{p.Lyap}
Let $\mu$ be an ergodic invariant probability measure of a
diffeomorphism $f$ of a compact manifold $M$. Fix a
$C^1$-neighborhood $\cU$ of $f$, a neighborhood $\cV$ of $\mu$ in
the space of probability measures with the weak topology, a
Hausdorff-neighborhood $\cK$ of the support of $\mu$, and a
neighborhood $O$ of $L(\mu)$ in $\RR^d$. Then there is $g\in\cU$
and a periodic orbit $\gamma$ of $g$ such that the Dirac measure
$\mu_\gamma$ associated to $\gamma$ belongs to $\cV$, its support
belongs to $\cK$, and its Lyapunov vector $L(\mu_\gamma)$ belongs
to $O$.
\end{prop}

\begin{proof}[Proof of Theorem \ref{theo3}]
The proof is similar to the proof of Theorem~\ref{ergodlemma}.
Note first that it is enough to prove the Theorem restricted to
a small $C^1$-neighborhood $\cU$ of an arbitrary diffeomorphism $f_0\in \diff$.
In particular, one may assume that $\log \|Df\|$ and $\log \|Df^{-1}\|$
are bounded by some constant $S>0$ for any $f\in \cU$.

Let $X$ be the space of triples $(\mu,K,L)$ where $\mu$ is a
probability measure on $M$, $K\subset M$ is a compact set, and
$L\in [-S,S]^d$, endowed with the product topology of the weak
topology on the probability measures, the Hausdorff topology on
the compact subspaces of $M$, and the usual topology on $\RR^d$.

To any periodic orbit $\gamma$ of a diffeomorphism $f$ we
associate a triple $x_\gamma = (\mu_\gamma,\gamma, L(\mu_\gamma))$.
We denote by $X_f$ the closure of the set $\{x_\gamma, \gamma\in Per(f)\}$.
This is a compact subset of $X$, and hence an element of the space $\cK(X)$ of compact subsets of $X$
endowed with the Hausdorff topology.

One easily verifies that the map $f\mapsto X_f$ is lower
semi-continuous on the set of Kupka-Smale diffeomorphisms, which
is residual in $\cU$. As a consequence, this map is continuous
on a residual subset $\cR\cap \cU$ of the set of  Kupka-Smale
diffeomorphisms, hence of $\cU$.

Consider $f\in \cR\cap \cU$ and $\mu$ an ergodic probability measure of
$f$. Proposition~\ref{p.Lyap} allows us to create a periodic
orbit $\gamma$ such that $x_\gamma$ is arbitrarily close to $(\mu,\supp(\mu),L(\mu))$; a small
perturbation makes this periodic orbit hyperbolic, and hence
persistent by perturbations; a new small perturbation yields a
Kupka-Smale diffeomorphism. Since $f$ is a continuity
point of $g\mapsto X_g$ in the set of Kupka-Smale diffeomorphisms, one has shown that
$(\mu,\supp(\mu),L(\mu))$ belongs to $X_f$, which implies the theorem.
\end{proof}

\subsection{Approximation by perturbation: proof of Proposition~\ref{p.Lyap}}
We fix an ergodic measure $\mu$ of a diffeomorphism $f$.
Let $\lambda_1<\cdots<\lambda_k$ be the Lyapunov exponents of $\mu$
and for every $i$ let $d_i$ be the multiplicity of the Lyapunov
exponent $\lambda_i$.

We consider a regular point $x$ for $\mu$ in the following sense:

\begin{itemize}

\item The probability measures $\frac 1n \sum_0^{n-1}\delta_{f^i(x)}$ and
$\frac 1n \sum_{0}^{n-1}\delta_{f^{-i}(x)}$ converge to $\mu$
as $n\to +\infty$.

\item $x$ has well-defined Lyapunov exponents and its exponents are
those of $\mu$. Moreover there is a splitting $T_xM=E_1\oplus \cdots \oplus
E_k$, such that:

\begin{itemize}
\item $dim(E_i)=d_i$;

\item the number $\frac1n \log(\|Df^n(u)\|)$ converges uniformly to
$\lambda_i$ on the set of unit vectors $u$ of $E_i$ as  $n$ tends to
$+ \infty$;

\item the angle between between the Lyapunov spaces $Df^n(E_i)$ and $Df^n(E_j)$ decreases at most subexponentially: $$\lim_{n \to +\infty}\frac1n\log \sin (\angle{E_i}{E_j})=0.$$
\end{itemize}

\item $x$ is well closable: for any $C^1$-neighborhood $\cU$ of $f$, any $\varepsilon>0$   any $N>0$ there is $n>N$ and $g\in \cU$ such that $x$ is periodic of period $n$ for $g$ and $d(g^i(x),f^i(x))<\varepsilon$ for $i\in \{1,\dots,n\}$.
\end{itemize}

The set of regular points for $\mu$ has full measure for $\mu$,
according to the Birkhoff ergodic theorem, the Oseledets
subadditive theorem, and Ma\~n\'e's ergodic closing lemma. In
particular, such a point $x$ exists.
Fix a local chart at $x$ such that $E_i$ coincides with the space
$$E_i=\{0\}^{\sum_1^{i-1}d_{\ell}}\times \RR^{d_i}\times
\{0\}^{\sum_{i+1}^k d_{\ell}}$$
and a Riemannian metric on $M$ which coincides with the Euclidian metrics on this local chart. For
$i\leq j$ we denote  $E_{i,j}=E_i\oplus \cdots \oplus E_j$.

Given  a number $C>0$ and two linear subspaces $E,F\subset T_x M$
having the same dimension, we will say that \emph{the inclination
of $F$ with respect to $E$ is less than $C$} if $F$ is transverse
to the orthogonal space $E^\perp$ and if $F$ is the graph of a
linear map $\varphi \colon E\to E^\perp$ of norm bounded by $C$.
\medskip

We divide the proof of Proposition~\ref{p.Lyap} into two main steps
stated now.
In the first step we build the perturbation, and in the second
step we verify the announced properties.

\begin{lemma}\label{l.perturbation} For every $C^1$-neighborhood $\cU$ of $f$, and any $\varepsilon>0$ there are:
\begin{itemize}
\item a number $C>0$,
\item a sequence $(\varepsilon_n)$ of positive numbers with $\lim_{n\to\infty} \varepsilon_n =0$,
\item a sequence of integers $t_n\to+\infty$,
\item a sequence of linear isometries $P_n\in O(\RR,d)$  such that
$\|P_n-Id\|<\varepsilon$,
\item a sequence of diffeomorphisms $f_n\in \cU$,
\end{itemize}
 with the following properties:
\begin{itemize}
\item[a)] The point $x$ is periodic of period $t_n$ for $f_n$.
\item[b)] The distance $d(f^t(x), f_n^t(x))$ remains bounded by $\varepsilon_n$ for $t\in\{0,\dots,t_n\}$.
In particular the point $f^{t_n}(x)$ belongs to the local chart we fixed  at $x=f^{t_n}_n(x)$.
This allows us to consider the derivative $Df^{t_n}(x)$ as an element of $GL(\RR,d)$.
\item[c)] The expression of $Df_n^{t_n}(x)$ in the local coordinates at $x$ coincides with $P_n\circ Df^{t_n}(x)$.
\item[d)] For every $i\leq j\in\{1,\dots,k\}$ the  inclination of $Df_n^{t_n}(x).E_{i,j}$ with respect to $E_{i,j}$ is less than $C$.
\end{itemize}
\end{lemma}

As $x$ is a regular point of $\mu$ one gets that the Dirac
measures along the (periodic) orbits of $x$ for $f_n$ converge
weakly to $\mu$, and that the orbits themselves converge to the
support of $\mu$ in the Hausdorff topology as $n\to \infty$. Then
we conclude the proof of Proposition~\ref{p.Lyap} by proving:

\begin{lemma}\label{l.Lyap} The Lyapunov vectors of the orbits of $x$ by $f_n$ converge to the Lyapunov vector of $\mu$ for $f$ when $n\to+\infty$.
\end{lemma}

\subsection{Building the perturbations: Proof of Lemma~\ref{l.perturbation}}

We cover $M$ by finitely many local charts $\varphi_i\colon V_i\to
\RR^d$ and we choose open subsets  $W_i\subset V_i$,
relatively compact in $V_i$, such that the $W_i$ cover $M$.
For every
$t\in\ZZ$ we fix $i(t)$ such that $f^t(x)\in W_{i(t)}$.

Shrinking $\varepsilon$ if necessary we may assume that:
\begin{itemize}
\item every $10\varepsilon$ perturbation of $f$ in $\diff$ is contained in $\cU$,
\item $\varepsilon$ is smaller than the infima of the distances between $W_i$ and the complement of $V_i$.
\end{itemize}
We fix $\varepsilon_0\in (0,\varepsilon)$  such that for any
$y\in W_i$ and any point $z$ such that $d(y,z)<\varepsilon_0$
then:
\begin{itemize}
\item The distance $d(f(y),f(z))$ is smaller than $\varepsilon$. In particular given any $j$ with $f(y)\in W_j$ one has $z \in V_i$ and $f(z)\in
V_j$.

\item There is a linear map $A\colon T_z M\to T_{f(z)}M$ such that $\|A-Df(z)\|<\varepsilon$ and such that the expressions of $A$ and of $Df(y)$ in the charts $V_i$ and $V_j$ coincide.

\item If $g$ is a $\varepsilon_0$-$C^1$-perturbation of $f$ then  $d(g(y),f(y))<\varepsilon$, so that $g(y)\in V_j$. Furthermore,  there is a linear
map $B\colon T_y M\to T_{g(y)}M$ such that
$\|B-Dg(y)\|<\varepsilon$ and such that the expressions of $B$ and
of $Df(y)$ in the charts $V_i$ and $V_j$ coincide.
\end{itemize}
We fix now a sequence $0<\varepsilon_i <\varepsilon_0$ decreasing
to $0$. As the point $x$ is well closable, there is a sequence
of $\varepsilon_n$-perturbations $h_{n}$ of $f$ and integers $t_n\in
\NN$ with $t_n\to +\infty$ such that:
\begin{itemize}
\item the point $x$ is periodic of period $t_{n}$ for $h_{n}$;
\item the distance $d(f^t(x),h_{n}^t(x))$ remains bounded by $\varepsilon_{n}$
for $t\in \{0,\dots,t_{n}\}$.
\end{itemize}
The diffeomorphism $f_{n}$ built below will preserve the orbit of $x$ by $h_{n}$, and hence items a) and b) of the lemma will be satisfied.
\medskip

As $h_n$ is $\varepsilon_n$-close to $f$, for every
$t\in\{0,t_n-1\}$ the map $B_t\colon T_{h_n^t(x)}M\to
T_{h_n^{t+1}(x)}M$ whose expression in the coordinates $V_{i(t)},
V_{i(t+1)}$ is $Df(h_n^t(x))$ satisfies
$\|B_t-Dh_n(h_n^t(x))\|<\varepsilon$. As
$d(f^t(x),h_n^t(x))<\varepsilon_n$, the linear map $A_t\colon
T_{h_n^t(x)}M\to T_{f(h_n^{t}(x))}M$ whose expression in the
coordinates $V_{i(t)}, V_{i(t+1)}$ is $Df(f^t(x))$ satisfies
$\|A_t-Df(h_n^t(x))\|<\varepsilon$.
One deduces
$$\| Dh_{n}(h^t_{n}(x))- Df^(f^t(x))\| <2\varepsilon.$$
By Franks Lemma, there is a diffeomorphism $g_n$ such that $g_n=
h_n$ on the periodic orbit of $x$, $g_n$ is
$3\varepsilon$-$C^1$-close to $h_n$, $g_n=h_n$ out of an
arbitrarily small neighborhood of the orbit of $x$ and such that
the expression of $Dg$ at the point $g_n^t(x)=h_n^t(x)$ in the
coordinates $V_{i(t)}, V_{i(t+1)}$ is  $A_{t}$, i.e. the same as $Df(f^t(x))$.

As a consequence $g_n$ is a $4\varepsilon$-perturbation of $f$ satisfying:
\begin{itemize}
\item $g_{n}$ preserves the orbit of $x$ by $h_{n}$,
\item the expression of $Dg_n^{t_n}(x)$ in the local coordinates $V_{i(0)}$ is the same as $Df^{t_n}(x)$.
\end{itemize}

In order to conclude the proof of Lemma~\ref{l.perturbation}
we need only to control the inclinations. For that we will prove

\begin{claim}Given any $\eta>0$ and given an integer $l>0$ there is $C >0$ such that,
given any pair of $l$-uples  $(F_1,\dots, F_l),(G_1,\dots,G_l)$ of
vector subspaces of $\RR^d$ such that $dim \; F_j= dim \; G_j$ for
all $j\in\{1,\dots,l\}$, there is an orthogonal matrix $P$ such
that $\| P-id\|<\eta$, and the inclination of $P(G_j)$ with
respect to $F_i$ is less that $C$.
\end{claim}
\begin{proof} The proof is done by induction on $l$. Assuming the
result obtained for $l-1$ and $\eta/2$, we perform a very small
perturbation of the matrix $P$ to control the inclination of
$P(G_l)$ with respect to $F_l$ while keeping the other
inclinations smaller than $2C$.
\end{proof}

Let $K>0$ be a bound on $\|Dg\|$ for any $g\in \cU$
and fix $\eta\in(0,\varepsilon K^{-1})$.
There exists $C>0$ such that the claim is satisfied for any $\ell=d^2$.
Hence, there exists $P_{n}\in O(\RR, d)$ with $\|P_{n}-id\|<\eta\leq\varepsilon$
such that $P_{n}Dg_n(g_n^{t_n-1}(x))$ is an
$\varepsilon$-perturbation of $Dg_n(g_n^{t_n-1}(x))$.
Now,
applying once more Franks Lemma, we obtain a sequence $f_n$
satisfying:
\begin{itemize}
\item $f_n$ is a $6\varepsilon$-perturbation of $f$, and hence belongs to $\cU$,
\item $f_{n}$ preserves the orbit of $x$ by $h_{n}$, and hence satisfies items a) and b) of the lemma,
\item the expression of $Df_n^{t_n}(x)$ in the local coordinates $V_{i(0)}$ is $P_n\circ Df^{t_n}(x)$,
\item for $i\leq j\in\{1,\dots,k\}$ the inclination of $Df_n^{t_n}(x).E_{i,j}$ with respect to $E_{i,j}$ is less than $C$.
\end{itemize}
This ends the proof of Lemma~\ref{l.perturbation}

\subsection{Lyapunov exponents: Proof of Lemma~\ref{l.Lyap}}
We consider the Lyapunov spaces $E_1, \dots, E_k$ of $x$ and for
every $j\in\{1,\dots,k-1\}$ we denote $F_j=E_1\oplus \cdots\oplus
E_j=E_{1,j}$ and $G_j=E_{j+1}\oplus\cdots \oplus E_k=E_{j+1,k}$.
Recall that $m(A)$ denotes the minimal expansion of a linear automorphism $A\in GL(\RR,d)$.

\begin{lemma} For any $\nu>0$ there is $n_\nu \geq 1$ such that for any $n\geq n_\nu$ and $j\in\{1,\dots, k-1\}$ one has:
\begin{equation}\begin{array}{lcrclcr}
\frac1n\log(\|Df^n|_{F_i}\|)&\leq&\lambda_j+\frac 12 \nu&\mbox{ and } &\frac1n\log(m(Df^n|_{G_i}))&\geq&\lambda_{j+1}-\nu,\\
\end{array}
\end{equation}
\begin{equation}\begin{array}{lcrclcr}
\frac1n\log(m(Df^{-n}|_{F_i}))&\geq &-\lambda_j-\frac 12
\nu&\mbox{ and }& \frac 1n
\log(\|(Df^{-n}|_{G_i}\|)&\leq&-\lambda_{j+1}+\nu.
\end{array}
\end{equation}
\end{lemma}
\begin{proof}
This is an easy consequence of Oseledets theorem: the rate of expansion
on the Lyapunov space $E_i$ converges uniformly to the Lyapunov
exponents by positive and negative iterations, together with the
fact that the angles between the images of the Lyapunov spaces
decrease subexponentially with the number of iterations.
\end{proof}

For $K>0$, let $\cC^u_{j,K}$ be the cone of vectors whose inclination with respect to $G_j$
is smaller than $K$:
$$\cC^u_{j,K}=\{v=v^s+v^u\in T_xM : v^s\in F^j, v^u\in G_j, \|v^s\|\leq  K\|v^u \|\}.$$
We denote by $\cC^s_{j,K}$ the closure of $T_xM \setminus \;
\cC^u_{j,K}.$ Note that, one has
$Df_n^{t_n}(G_j)\subset \cC^u_{j,C}\subset \cC^u_{j,2C}\subset
\cC^u_{j,4C}$.

\begin{lemma}\label{l.Lyap2}
For every $\nu>0$ there is $n'_{\nu}>0$ such that for any $n\geq n'_{\nu}$
and $j\in\{1,\dots,k\}$ one has:
\begin{itemize}
\item The cone $\cC^u_{j,4C}$ is strictly invariant; more precisely:
$$Df_n^{t_n}(\cC^u_{j,4C})\subset \cC^u_{j,2C}\subset
\cC^u_{j,4C}.$$ As a consequence, the cone $\cC^s_{j,4C}$ is
strictly invariant by $Df_n^{-t_n}$.

\item For every unit vector $v\in \cC^u_{j,4C}$  one has $$\frac
1{t_n}\log\|Df_n^{t_n}(v)\|\geq \lambda_{j+1}-\nu.$$

\item For every unit vector $w\in \cC^s_{j,4C}$  one has $$\frac
{-1}{t_n}\log\|Df_n^{-t_n}(w)\|\leq \lambda_{j}+\nu.$$

\end{itemize}
\end{lemma}
\begin{proof}
Consider  $v=v^s+v^u\in \cC^u_{j,4C}$, $v^s\in F_j$ and $v^u\in
G_j$. By definition we have $\|v^s\|\leq 4C\|v^u\|$.
For $n\geq n_\nu$ we get from equation (1) that:
$$\begin{array}{rcl}
\|Df^n(v^s)\|&\leq &e^{n(\lambda_{j}+\frac12 \nu)}\|v^s\|\\
&\leq& 4C e^{n(\lambda_{j}+\frac12 \nu)}\|v^u\|\\
&\leq& 4C e^{n(\lambda_j-\lambda_{j+1}+\nu)}\|Df^n(v^u)\|.
\end{array}
$$
Hence,
$$
\frac{\|Df^n(v)\|}{\|v\|} \geq \frac{\|v^u\|}{\|v\|}
\frac{\|Df^n(v^u)\|-\|Df^n(v^s)\|}{\|v^u\|} \geq \frac
1{\sqrt{1+16C^2}} (1-4C e^{n(\lambda_j-\lambda_{j+1}+\nu)})
e^{n(\lambda_{j+1}-\frac\nu 2)}.$$
Notice that $4C e^{n(\lambda_j-\lambda_{j+1}+\nu)}$ tends to $0$
when $n\to+\infty$. In particular for $n$ large one has:
$$\inf\left\{\frac1n \log\frac {\|Df^n(v)\|}{\|v\|}, v\in\cC^u_{j,4C}\right\}\geq \lambda_{j+1}-\frac34\nu.$$
Recall that the expression in the chart at
$x$ of  $Df^{t_n}_n$ is the same as $P_n\circ Df^{t_n}$, where
$P_n$ is an isometry. It follows that for $n$ large enough one has
$$\inf\left\{\frac1n \log\frac {\|Df^{t_n}_n(v)\|}{\|v\|}, v\in\cC^u_{j,4C}\right\}\geq \lambda_{j+1}-\nu.$$
Furthermore, $P_n$ has been chosen in such a way that
$Df^{t_n}_n(v^u)$ belongs to the cone $\cC^u_{j,C}$. As a
consequence, for $4C e^{t_n(\lambda_j-\lambda_{j+1}+\nu)}$ small
enough the vectors $Df^{t_n}_n(v)=Df^{t_n}_n(v^u)+Df^{t_n}_n(v^s)$
belong to $\cC^u_{j,2C}$, for all $v\in \cC^u_{j,4C}$.
This proves the two first items of the Lemma~\ref{l.Lyap2}.
\medskip

Consider now
$w=w^s+w^u\in \cC^s_{j,4C}$ with $w^s\in F_j$ and $w^u\in G_j$. By
hypothesis one has $\|w^u\|\leq\frac1{4C}\|w^s\|$.
Let us decompose $\tilde w:=P^{-1}(w)$ as
$\tilde w=\tilde w^s+\tilde w^u$ with $\tilde w^s\in F_j$,$\tilde w^u\in G_i$.
Since $\|P_n^{-1}-id\|=\|P_n-id\|<\varepsilon$, one deduces that
$$\|\tilde w^u\|\leq
\frac{\frac1{4C}+\varepsilon}{1-\varepsilon}\|w^s\|.$$

We denote by $\bar w^s$ and $\bar w^u$ the vectors of
$T_{f^{t_n}(x)}M$ whose expressions in the local coordinates at
$x$ are equal to those of $\tilde w^s$ and $\tilde w^u$,
respectively. Note that, by construction, $Df_n^{-t_n}(w)=
Df^{-t_n}(\bar w^u)+ Df^{-t_n}(\bar w^s)$.
The proof of the third item consists now in estimating and comparing the norms
$\|Df^{t_n}(\bar w^u)\|$ and $\|Df^{t_n}(\bar w^s)\|$ using
equation (2) instead of equation (1), in a similar way as above.
\end{proof}
\bigskip

Let us now end the proof of Lemma~\ref{l.Lyap}.

\begin{proof}[Proof of Lemma~\ref{l.Lyap}]
Fix $\nu$ smaller than $\frac 1{10}\inf_{i\neq
j}\{|\lambda_i-\lambda_j|\}$ and consider $n>n'_{\nu}$. Then Lemma~\ref{l.Lyap2}
implies:

\begin{itemize}
\item $Df_n^{t_n}(x)$ admits a (unique) invariant vector space $G^n_{i}$
of dimension $\dim(G_i)$ in $\cC^u_{i,4C}$.

\item The restriction of $Df_n^{t_n}(x)$ to $G^n_{i}$ has a minimal
dilatation larger than $\lambda_{i+1}-\nu$.

\item $Df_n^{t_n}(x)$ admits a (unique) invariant vector space $F^n_{i}$
of dimension $\dim(F_i)$ in $\cC^s_{i,4C}$.

\item The restriction of $Df_n^{t_n}(x)$ to $F^n_{i}$ has norm
smaller than $\lambda_{i}+\nu$.

\end{itemize}
Set $E^n_{i} := F^n_i \; \cap \; G^n_{i-1}$. It is a vector space of
dimension at least $\dim(F^n_i) + \dim(G^n_{i-1}) - d = \dim(E_i)$. Furthermore,
one has
$$\lambda_i-\nu\leq m(Df_n^{t_n}(x)|_{E^n_i})\leq\|Df_n^{t_n}(x)|_{E^n_i}\|\leq \lambda_{i}+\nu.$$
As $\lambda_{i}+\nu<\lambda_{i+1}-\nu$, one deduces
that the sum $E^n_1+ \dots + E^n_k$ is a direct sum.
It follows that $\dim(E^n_i) \leq \dim(E_i)$, and hence $\dim(E^n_i) = \dim(E_i)$. Hence $x$ has $\dim(E_i)$ Lyapunov exponents contained in $[\lambda_i-\nu,\lambda_i+\nu]$.
This proves that for $n$ large the Lyapunov vector of the measure associated to
the $f_n$-orbit of $x$ is $\nu$-close to the
Lyapunov vector of $\mu$, ending the proof of
Lemma~\ref{l.Lyap}.
\end{proof}

\section{Generic Nonuniform Hyperbolicity}\label{prelimhyp}

In this section we obtain the nonuniform hyperbolicity of the
generic measures over an isolated transitive set (items~(b.iv) and (b.v)  of Theorem \ref{theo4}),
and also of the generic ergodic measures of $C^1$-generic diffeomorphisms
(item (ii) of Theorem \ref{theo2}).
We also give the proof of Corollary \ref{cor.bgv}
which approximates an ergodic measure by period measures whose Lyapunov exponents are almost constant
on the bundles  of the finest dominated splitting.



\subsection{Approximation by periodic orbits with mean Lyapunov exponents}

Since it is very similar to the proofs of Lemma \ref{ergodlemma} and Theorem \ref{theo3},
we now only sketch out the proof of Corollary \ref{cor.bgv}.
This uses  \cite{BGV}, which constructs perturbations on sets of
periodic orbits which exhibit a lack of domination.
In our context we may state this tool in the following way:

\begin{theorem}[\cite{BGV}] \label{bgv}
Let $\{\gamma_k\}$ be a family of hyperbolic periodic orbits of $f \in \diff$ and $F_1 \oplus_< \cdots \oplus_< F_k$ be the finest dominated splitting over $\overline{\cup_{k \in \NN} \; \gamma_k}$. Assume that there is no infinite subset $\Gamma$ of $\cup_{k \in \NN} \; \gamma_k$ such that the finest dominated splitting over $\overline{\Gamma}$ is strictly finer than $F_1 \oplus_< \cdots \oplus_< F_k$. Then given any $\varepsilon > 0$ there is an $\varepsilon$-perturbation $g$ of $f$ such that $g$ exhibits a periodic orbit, coinciding with one of the original orbits, and whose Lyapunov exponents inside each bundle $F_i$ all coincide.
\end{theorem}

\begin{proof}[{Proof of Corollary \ref{cor.bgv}}]

Consider $X$ the space of triples $(\mu,K,L)$ where $\mu$ is a
probability measure on $M$, $K\subset M$ is a compact set, and
$L $ is a vector in  $\RR^d$, endowed with the usual product topology. If $\gamma$ is a periodic orbit, we denote by $x_\gamma$ the
triple $(\mu_\gamma,\gamma,L(\mu_\gamma))$ (the measure associated to $\gamma$, its support $\gamma$, and its Lyapunov vector $L(\mu_\gamma)$).
As in the proof of Theorem~\ref{theo3}, the map $f \mapsto X_f$, which to each diffeomorphism $f$ associates the closure $X_f$ of the set
$\{x_\gamma, \gamma\in Per(f)\}$, is continuous on a residual subset $\cG$ of $\diff$.

Consider now, for such a $C^1$-generic $f$, a triple of the form $(\mu, supp(\mu), v)$, where $\mu$ is a generic (and hence ergodic) measure supported in $\supp(\mu)$ and $v$ is the vector given by

$$\begin{array}{crl}
v =& \left\{\frac{\int \log \|det \, Df|_{F_1}\| \; d{\mu}}{\dim(F_1)}\right\}^{\dim(F_1)} &\times \left\{\frac{\int \log \|det \, Df|_{F_2}\| \; d{\mu}}{\dim(F_2)}\right\}^{\dim(F_2)}\times  \ldots\\
& \dots &\times \left\{\frac{\int \log \|det \, Df|_{F_k}\| \; d{\mu}}{\dim(F_k)}\right\}^{\dim(F_k)},
\end{array}$$
where the $F_i$ are the bundles of the finest dominated splitting on $\supp(\mu)$.

We claim that $(\mu, \supp(\mu), v) \in X_f$, which proves Proposition \ref{cor.bgv}.
By Theorem \ref{theo3}, there is a sequence of periodic orbits $\gamma_k$ such that $(\mu_{\gamma_k}, \gamma_k, L(\gamma_k))$ accumulate on $(\mu, \supp(\mu), L(\mu))$. Since these orbits Hausdorff-accumulate on $\supp(\mu)$, it follows that for large enough $K$ the set $\{\gamma_k\}_{k \geq K}$ admits as its finest dominated splitting a continuation of the dominated splitting $F_1 \oplus_< \ldots \oplus_< F_k$ over $\supp(\mu)$, so that no subsequence of $\{\gamma_k\}_{k \geq K}$ admits a finer dominated splitting. Now an application of Theorem~\ref{bgv} yields after a small perturbation a periodic orbit $\gamma'$ whose Lyapunov exponents inside each $F_i$ all coincide. Up to performing a new perturbation we obtain a triple $(\mu_{\gamma''},\gamma'',L(\mu_{\gamma''}))$ close to $(\mu, \supp(\mu), v)$ for some $C^1$-generic $g\in\cG$ arbitrarily close to $f$.
Since $f$ is a continuity point of $f\mapsto X_f$, one gets that $(\mu, \supp(\mu), v) \in X_f$, ending the proof.
\end{proof}

\subsection{Proof of  Theorem \ref{theo4}, items (b.iv) and (b.v)}\label{b.ivsection}
In \cite{BocV} arguments involving flags are used to obtain
semicontinuity properties of the Lyapunov exponents and Lyapunov
spaces \emph{when the diffeomorphism $f$ varies} and keeping
constant a volume measure $\mu$ on $M$. An application of the
Semicontinuity Lemma then shows that $C^1$-generic (conservative)
diffeomorphisms are continuity points for the set of Lyapunov
exponents and their corresponding Lyapunov spaces.

In our dissipative setting, identical arguments yield semicontinuous
variation of the exponents \emph{when the measure $\mu$ varies} and
keeping the diffeomorphism $f$ fixed. The Semicontinuity Lemma then
yields that generic measures are continuity points for the Lyapunov
exponents. That is, we have:

\begin{proposition}\label{BocV}

Given $\La$ a compact invariant set of a diffeomorphism $f$, then
there is a residual subset $\CS^*$ of $\mflerg$ which consists of
ergodic measures $\mu$ which are continuity points for the map
$$\Phi: \mathcal{M}_f^{erg}(\La) \rightarrow \RR^d$$
$\hspace{210pt} \mu \; \mapsto L(\mu) ,$

\noindent where  $L(\mu)=(\la_1^{\mu}, \ldots, \la_d^{\mu})$ denotes the Lyapunov vector of
$\mu$.
\end{proposition}

\begin{remark}
Here we state the continuity restricted to the ergodic measures
simply because that makes it easier to state the continuity;
furthermore, we shall only use the continuity on the set of ergodic
measures.
\end{remark}

We are now ready to prove the
hyperbolicity of generic measures over isolated transitive sets of generic diffeomorphisms.
In fact, we will prove something stronger:

\begin{proposition}\label{propb.iv}
Let $\La$ be an isolated transitive set of a $C^1$-generic diffeomorphism
$f$, with finest dominated splitting $F_1 \oplus_< \ldots \oplus_< F_k$
over $T_{\La}M$. Then there is a residual subset $\cS$ of $\mfla$
such that for any measure $\mu \in \cS$ and any $i \in \{1, \ldots, k\}$
there is only one Lyapunov exponent $\la_i$ of $\mu$ in $F_i$, which
furthermore is non-zero.
\end{proposition}

\begin{remark}
The Proposition above shows that even if $\La$ is nonhyperbolic, and
thus contains periodic orbits of distinct indices (see \cite{BDPR}),
the generic hyperbolic measures it supports may all have the same
index. Indeed, by the Proposition the indices of the generic $\mu$'s
are restricted by the (dimensions of the) bundles of the finest
dominated splitting over $\La$. There are examples of
nonhyperbolic robustly transitive sets -- and hence of isolated
transitive sets of $C^1$-generic diffeomorphisms -- whose finest dominated
splitting  has only two bundles $E$ and $F$, see \cite{BonV}.
Thus in such examples all of the generic measures provided by
Proposition \ref{propb.iv} above must have the same index (namely,
the dimension of $E$), even though the set $\La$ is nonhyperbolic
and thus contains periodic orbits of distinct indices.
\end{remark}

\begin{proof}[Proof of Proposition \ref{propb.iv}]
Let $\La$ be a non-trivial isolated transitive set of a $C^1$-generic
diffeomorphism $f$.
Let us fix any bundle $F=F_i$ of the finest dominated splitting.
Given any $\mu \in \mfla$, we set

\begin{equation}\label{e.exp}
I(\mu)  :=  \int \log \|det \, Df|_F\| \; d{\mu}.
\end{equation}
Note that since $F$ is a continuous bundle, $I(\mu)$ varies
continuously with $\mu$ in the weak topology.

On the other hand, if $\mu \in \mflerg$ then
$$I(\mu) = \sum_{\tf(\la_i) \subset F} \la_i \, \dim(\tf(\la_i)), $$
where $\lambda_i$ and $F(\lambda_i)$
are respectively the Lyapunov exponents and the Lyapunov spaces inside $F$.

For any periodic measure $\mu_\gamma$, Franks lemma allows one to perturb
the diffeomorphism in $\diff$ so that
each sum $\lambda_i+\dots+\lambda_j$ for $1\leq i\leq j \leq d$ is different from zero.
An easy genericity argument hence implies that under a $C^1$-genericity assumption on $f$,
the quantity $I(\mu)$ never vanishes on the periodic measures of $f$.

Now, by Theorem~\ref{ergodlemma}, there exists a dense set
$\cD\subset \mfla$, consisting of hyperbolic
periodic measures such that $I(\nu) \neq 0$ for every $\nu \in \cD$.
Since the integral $I(\mu)$ varies continuously with $\mu$, we
conclude that $I(\mu) \neq 0$ in an open and dense subset of
$\mfla$.

If $\mu \in \mfla$ is a generic measure, we know that it is
ergodic, that $\supp(\mu)=\La$, that $I(\mu) \neq 0$, and that
(by proposition~\ref{BocV})
it is a continuity point for the map $\nu\to L(\nu)$ defined on $\mflerg$.
Using Corollary~\ref{cor.bgv} there is a sequence of periodic orbits $\gamma_{\ell}$
such that $L_{|F}(\mu_{\gamma_{\ell}})$ converges to some single
value $\la_F$. Since $\mu$ is a continuity point for $\nu\to L_{|F}(\nu)$ it
follows that $\la_F$ is the only Lyapunov exponent of $\mu$ in
$F$: this proves that the finest dominated splitting on $\supp(\mu)$ coincides with
the Oseledets splitting of $\mu$.
Moreover we must have
$$\la_F  = \frac{I(\mu)}{\dim(F)} \neq 0,$$
implying that $\mu$ is nonuniformly hyperbolic.
\end{proof}

\subsection{Proof of Theorem \ref{theo2}}\label{theo2section}
The argument is very similar to the proof of Proposition \ref{propb.iv}.
Since $f$ is $C^1$-generic, then for each periodic orbit $\gamma$,
the sum $\lambda_i+\dots+\lambda_j$ for $1\leq i\leq j \leq d$ is different from zero,
where $\lambda_1,\dots,\lambda_d$ denote the Lyapunov exponents of $\gamma$ with multiplicities.

Any generic ergodic measure $\mu$ is a continuity point
of the map $\mu\mapsto \supp(\mu)$ on $\mflerg$.
As a consequence the finest dominated splitting $F_1\oplus_{<}\dots\oplus_{<} F_k$
on $\supp(\mu)$ extends to the support of any ergodic measure $\nu$ close to $\mu$
in the weak topology. In particular any bundle $F=F_i$ of the finest splitting
extends to $\supp(\nu)$ and the map $\nu\mapsto I(\nu)$,
giving the sum of the Lyapunov exponents of $\nu$ inside $F$, varies continuously with $\nu$
on a neighborhood of $\mu$.
By Theorem~\ref{ergodlemma}, there exists a sequence of periodic measures $\mu_\gamma$
which converges to $\mu$ and such that $I(\mu_\gamma)\neq 0$. Since $\mu$ is generic, one
thus gets $I(\mu)\neq 0$.

By Proposition~\ref{BocV}, $\mu$
it a continuity point for the map $\nu\to L(\nu)$ defined on $\mflerg$.
Using Corollary~\ref{cor.bgv} the periodic measures may be chosen so that
the Lyapunov exponents in $L_{|F}(\mu_{\gamma_{\ell}})$ converge to some single
value $\la_F$. It follows that $\la_F$ is the only Lyapunov exponent of $\mu$ in
$F$: this proves that the finest dominated splitting on $\supp(\mu)$ coincides with
the Oseledets splitting of $\mu$. The argument proves that $\lambda_F$
is non-zero, and hence that $\mu$ is nonuniformly hyperbolic. 
\section{Invariant Manifolds for Dominated Hyperbolic Measures}
\label{pesinsection}

In this section we will prove a stronger version of  Theorem \ref{theo5} stated in Proposition~\ref{p.invariant-manifold}.
Fix a $C^1$-diffeomorphism $f$ of the manifold $M$ and an ergodic measure
$\mu$ whose support admits a dominated splitting $E\oplus_{_<}F$. One assumes that $E$ is non-uniformly contracted
for $\mu$ (i.e. the  Lyapunov exponents of $\mu$ in $E$ are all negative); notice that we do not assume that vectors in $F$ are (non-uniformly) expanded. We will prove the existence of stable manifolds tangent to $E$ for $\mu$-almost every point, and control the rate of approximation of the points in these stable manifolds.

\subsection{Adapted metrics}\label{ss.metric}

We first build an Euclidian metric on the tangent space at $\mu$-almost every point, depending in a measurable way on  the point, and which is adapted to
the tangent dynamics.

\begin{definition}
We say that a sequence $(A_n)$ of positive numbers \emph{varies
sub-exponentially} if for any $\eta > 0$, there exists a constant $C>0$ such that
$$C^{-1}.e^{-\eta.n}<A_n<C.e^{\eta.n}$$
\noindent for every $n \in \NN$.
\end{definition}

\begin{proposition}\label{p.metric}
Let $f$ be a $C^1$-diffeomorphism and $\mu$ be an
ergodic invariant probability measure.  Assume that there
is a $Df-invariant$ continuous subbundle $E\subset
T_{\supp(\mu)}M$ defined over the support of $\mu$. Let $\lambda_E^+$ be the
maximal Lyapunov exponent of the measure $\mu$ in $E$.

Then for any $\varepsilon>0$ there exists an integer $N\geq 1$ and a
measurable function $A$ from $M$ to $(0,+\infty)$ such that :
\begin{itemize}
\item the sequences $\left(A(f^n(x))\right)_{n\in \NN}$  and $\left(A(f^{-n}(x))\right)_{n\in \NN}$ vary
sub-exponentially for each $x\in M$;
\item if $\|.\|'_x$  denotes the metric on $E_x$ defined by
$$\|v\|'_x=\sum_{0\leq k< N} e^{-k.(\lambda_E^++\varepsilon)}.A(f^k(x)).\|D_xf^k.v\|,$$
then for  $\mu$-almost every point $x$, for every $v\in E_x$ one has
$$\|D_xf.v\|'_{f(x)}\leq e^{\lambda_E^++\varepsilon}.\|v\|'_x.$$
\end{itemize}
\end{proposition}

\begin{remark} Since the integer $N$ is uniformly bounded, the
(measurable) metric $\|.\|'$ is quasi-conformally equivalent to the
initial metric $\|.\|$.
\end{remark}

Before proving Proposition~\ref{p.metric}, let us first explain how the Lyapunov
exponents may be computed as a limit of Birkhoff sums given by the
derivative of $f$.
\begin{lemma}\label{l.birkhoff}
Let $f$ be a $C^1$-diffeomorphism, $\mu$ be an
ergodic invariant probability measure, and $E\subset T_{\supp(\mu)}M$ be a $Df$-invariant
continuous subbundle defined over $\supp(\mu)$. Let $\lambda^+_E$
be the upper Lyapunov exponent in $E$ of the measure $\mu$.

Then, for any $\varepsilon>0$, there exists an integer
$N_\varepsilon$ such that, for $\mu$-almost every point $x\in M$ and
any $N\geq N_\varepsilon$, the Birkhoff averages
$$\frac{1}{k.N}\sum_{\ell=0}^{k-1}\log\|Df^N_{|E}(f^{\ell.N}(x))\|$$
converge towards a number contained in
$[\lambda_E^+,\lambda_E^++\varepsilon)$, when $k$ goes to $+\infty$.
\end{lemma}
\begin{proof}
The exponent $\lambda_E^+$ is given by:
$$\lambda_E^+=\lim_{n\to +\infty} \frac{1}{n}\int \log \|Df^n_{|E}\|d\mu.$$
One fixes an integer $n_0\geq 1$ large enough so that for any $n\geq
n_0$ we have:
\begin{equation}\label{e.mean}
\left|\frac{1}{n}\int \log \|Df^n_{|E}\|d\mu-\lambda_E^+\right|\leq
\frac{\varepsilon}{2}.
\end{equation}
The measure $\mu$ is ergodic for the dynamics of $f$, but it may happen that $\mu$ is not ergodic for $f^{n_0}$. Hence, it
decomposes as
$$\mu=\frac{1}{m}\left(\mu_1+\dots+\mu_m\right),$$
where $m\in\NN\setminus\{0\}$ divides $n_0$ and each $\mu_i$ is an ergodic $f^{n_0}$-invariant measure
such that $\mu_{i+1}=f_\ast \mu_i$ for each $i$ $(\text{mod } m)$.
Let $A_1\cup\dots\cup A_m$ be a measurable partition of $(M,\mu)$
such that $f(A_i)=A_{i+1}$ for each $i$ $(\text{mod } m)$ and
$\mu_i(A_i)=1$.

Note that by~(\ref{e.mean}), there exists $i_0\in \{1,\dots,m\}$,
such that
\begin{equation}\label{e.mean2}
\frac{1}{n_0}\int \log \|Df^{n_0}_{|E}\|d\mu_{i_0}\leq \lambda_E^+ + \frac{\varepsilon}{2}.
\end{equation}

For $N\geq 1$, and $\mu$-a.e. point $x$, one decomposes the segment
of $f$-orbit of length $N$ of $x$ as $(x,f(x),\dots, f^{j-1}(x))$,
$(f^j(x),\dots, f^{j+(r-1).n_0-1}(x))$ and
$(f^{j+(r-1).n_0}(x),\dots, f^{N-1}(x))$ such that $j<n_0$,
$j+r.n_0\geq N$ and all the points $f^j(x),
f^{j+n_0}(x),\dots,f^{j+r.n_0}$ belong to $A_{i_0}$. One deduces
that

\begin{equation*}
\begin{array}{rl}
\|Df^{N}_{|E}(x)\|\leq \|Df^{j}_{|E}(x)\|.&\left(\|Df^{n_0}_{|E}(f^j(x))\|\right.\dots\\
&\left.\dots\|Df^{n_0}_{|E}(f^{j+(r-2).n_0}(x))\|\right).\|Df^{N-(j+(r-1).n_0)}_{|E}(f^{j+(r-1).n_0}(x))\|.
\end{array}
\end{equation*}

Hence, for $\mu$-almost every point one has:
$$\log\|Df^{N}_{|E}(x)\|\leq 2n_0.C_f+\sum_{s=0}^{r-2}\log\|Df^{n_0}_{|E}(f^{j+s.n_0}(x))\|,$$
where $C_f$ is an upper bound for both $\log \|Df\|$ and $\log
\|Df^{-1}\|$.

The point $f^j(x)$ is regular for the dynamics
$(\mu_{i_0},f^{n_0})$. One deduces that the average
$\frac{1}{k.n_0}\sum_{\ell=0}^{k-1}\log\|Df^{n_0}_{|E}(f^{j+\ell.n_0}(x))\|$ converges to
$\frac{1}{n_0}\int \log \|Df^{n_0}_{|E}\|d\mu_{i_0}$. Hence

$$
\lim_{k\to +\infty} \frac{1}{k.N}\sum_{\ell=0}^{k-1}\log\|Df^N_{|E}(f^{\ell.N}(x))\|
\leq \frac{2n_0.C_f}{N}+ \lim_{k\to
+\infty}\frac{1}{k.n_0}\sum_{\ell=0}^{k-1}\log\|Df^{n_0}(f^{j+\ell.n_0}(x))\|.
$$

Hence, choosing $N> \frac{4n_0.C_f}{\varepsilon}$ and using the inequality (\ref{e.mean2}), one gets

$$
\lim_{k\to +\infty} \frac{1}{k.N}\sum_{\ell=0}^{k-1}\log\|Df^N_{|E}(f^{\ell.N}(x))\|
< \lambda_E^++\varepsilon.
$$

One the other hand, using that the norms are sub-multiplicative, one gets

$$
\lim_{k\to +\infty} \frac{1}{k.N}\sum_{\ell=0}^{k-1}\log\|Df^N_{|E}(f^{\ell.N}(x))\|
\geq \lim_{n\to +\infty} \frac{1}{n}\log \|D f^n(x)\|=\lambda_E^+.
$$
\end{proof}

One now comes to the proof of Proposition~\ref{p.metric}: one
considers a constant $\varepsilon>0$ and an integer $N\geq 0$ given
by Lemma~\ref{l.birkhoff} such that at $\mu$-almost every
point, the Birkhoff averages for $f^N$ of the functions
$x\mapsto \frac{1}{N}\log\|Df^N_{|E}(x)\|$ converge towards some numbers in
$[\lambda_E^+,\lambda_E^++\varepsilon)$.
In particular the sequence
$\sum_{\ell=0}^{k-1}\log \|Df^{N}_{|E}(f^{\ell.N}(x))\|$ is bounded
by $k.N.(\lambda_E^++\varepsilon)$ when $k$ is large.

This allows us to define the quantity
\begin{equation}\label{e.A}
A(x)=\max_{k\geq 0} \left( e^{-k.N.(\lambda_E^++\varepsilon)}.\prod_{\ell=0}^{k-1}\|Df^{N}_{|E}(f^{\ell.N}(x))\|\right),
\end{equation}
with the convention $\prod_{\ell=0}^{k-1}\|Df^{N}_{|E}(f^{\ell.N}(x))\|=1$ for $k=0$. Note that $A(x)\geq 1$, by definition.

The Proposition~\ref{p.metric} now follows from the next two lemmas.
\begin{lemma}\label{l.iterate}
At $\mu$-almost every point $x$, the metric
$$\|v\|'_x=\sum_{0\leq j<N} e^{-j.(\lambda_E^++\varepsilon)}.A(f^j(x)).\|D_xf^j_{|E}.v\|,$$
on $E_x$ satisfies
$$\|D_xf.v\|'_{f(x)}\leq e^{\lambda_E^++\varepsilon}.\|v\|'_x.$$
\end{lemma}
\begin{proof}
We write :
\begin{equation*}\label{e.iterate}
\begin{split} 
\|D_xf.v\|'_{f(x)}=\sum_{j=0}^{N-2} e^{-j.(\lambda_E^++\varepsilon)}.A(f^{j+1}(x)).\|D_xf^{j+1}.v\|+ e^{-(N-1).(\lambda_E^++\varepsilon)}.A(f^N(x)).\|Df^{N}(v)\|\\
\leq e^{\lambda_E^++\varepsilon}.\sum_{j=1}^{N-1} e^{-j.(\lambda_E^++\varepsilon)}.A(f^j(x)).\|D_xf^{j}.v\|+ e^{-(N-1).(\lambda_E^++\varepsilon)}.A(f^N(x)).\|Df^{N}_{|E}\|.\|v\|.
\end{split}
\end{equation*}
Hence one obtains the required estimate from the following:

\begin{claim}
$$A(f^N(x)).\|Df^N_{|E}(x)\|\leq e^{N.(\lambda_E^++\varepsilon)}.A(x).$$
\end{claim}
\noindent The proof of the claim is the following computation:
$$\begin{array}{rl}
A(f^N(x))&=\max_{k\geq 0} \left(e^{-k.N.(\lambda_E^++\varepsilon)}.\prod_{\ell=0}^{k-1}\|Df^{N}_{|E}(f^{(\ell+1).N}(x))\|\right)\\
&= e^{N.(\lambda_E^++\varepsilon)}.\max_{k\geq 0} \left(e^{-(k+1).N.(\lambda_E^++\varepsilon)}.\prod_{\ell=1}^{k}\|Df^{N}_{|E}(f^{\ell.N}(x))\|\right)\\
&= e^{N.(\lambda_E^++\varepsilon)}.\max_{k\geq 1} \left(e^{-k.N.(\lambda_E^++\varepsilon)}.\prod_{\ell=1}^{k-1}\|Df^{N}_{|E}(f^{\ell.N}(x))\|\right).
\end{array}
$$
Hence
$$\begin{array}{rl}
A(f^N(x)).\|Df^N_{|E}(x)\|&=e^{N.(\lambda_E^++\varepsilon)}.\max_{k\geq 1} \left(e^{-k.N.(\lambda_E^++\varepsilon)}.\prod_{\ell=0}^{k-1}\|Df^{N}_{|E}(f^{\ell.N}(x))\|\right)\\
&\leq e^{N.(\lambda_E^++\varepsilon)}.\max_{k\geq 0} \left(e^{-k.N.(\lambda_E^++\varepsilon)}.\prod_{\ell=0}^{k-1}\|Df^{N}_{|E}(f^{\ell.N}(x))\|\right)\\
&=e^{N.(\lambda_E^++\varepsilon)}\cdot A(x).
\end{array}
$$
This ends the proofs of the claim and of Lemma~\ref{l.iterate}.
\end{proof}
\medskip

\begin{lemma}\label{l.slow}
At $\mu$-almost every point $x$, the sequences
$\left(A(f^n(x))\right)_{n\in \NN}$ and $\left(A(f^{-n}(x))\right)_{n\in \NN}$ vary
sub-exponentially.
\end{lemma}
\begin{proof}
For $k\in\NN$ we consider the Birkhoff sum $S_k$ of the function $x\mapsto
-N(\lambda_E^++\varepsilon)+\log\|Df^N_{|E}(x)\|$ relative to the dynamics of
$f^N$. For $\mu$-a.e. point $x$, the Birkhoff average $\frac{S_k(x)}k$ converges when
$k$ tends to  $+\infty$ towards a number $\lambda<0$. One deduces that for any small $\eta>0$, there exists
$C>0$ such that we have for any $k\in \NN$:
$$
(\lambda-\eta).k-C\leq S_k(x)\leq (\lambda+\eta).k+C.
$$
For any integer $n\geq 0$, one has $S_k(f^{n.N}(x))=S_{k+n}(x)-S_n(x)$ so
that:
$$
(\lambda-\eta)k-2\eta n-2C\leq S_k(f^{n.N}(x))\leq (\lambda+\eta)k+2\eta n+2C.
$$

In particular, using that $\lambda$ is negative and
$\eta<|\lambda|$, we get
\begin{equation}\label{e.eta}
0\leq \max_{k\geq 0}S_k(f^{n.N}(x)) \leq 2\eta n+2C.
\end{equation}
This implies the subexponentiality of the sequence $\left(A(f^{n.N}(x))\right)_{n\in \NN}$ since
$$\log A(f^{n.N}(x))=\max_{k\geq 0}S_k(f^{n.N}(x)).$$
The subexponentiality of the sequence $\left(A(f^n(x))\right)_{n\in \NN}$ follows from the subexponentiality of the sequence $\left(A(f^{n.N}(x))\right)_{n\in \NN}$.

We now show the subexponentiality of the sequence $\left(A(f^{-n.N}(x))\right)_{n\in \NN}$ for $\mu$-almost every point.
We first notice that, for $k>n$, one can decompose $S_k(f^{-n.N}(x))$ in $S_n(f^{-n.N}(x))+S_{k-n}(x)$. Hence we have
$$
\max_{k\geq 0}S_k(f^{-n.N}(x))\leq \max_{0\leq k\leq n}S_k(f^{-n.N}(x))+\max_{k\geq 0}S_k(x)
$$
The subexponentiality of the sequence $\left(A(f^{-n.N}(x))\right)_{n\in \NN}$  thus follows from the following claim:

\begin{claim} For any $\eta>0$, there is a contant $C>0$ such that for any $n\geq 0$,
$$
0\leq {\max}_{0\leq k\leq n} S_k(f^{-n.N}(x)) \leq 2\eta n+2C.
$$
\end{claim}
For proving the claim, we consider the Birkhoff sum $\widetilde {S_k}$ of  the function $x\mapsto
-N(\lambda_E^++\varepsilon)+\log\|Df^N_{|E}(x)\|$ for the dynamics of
$f^{-N}$. For $\mu$-a.e. point $x$, the Birkhoff averages $\frac{\widetilde{S_k}(x)}k$ and $\frac{S_k(x)}k$ converges (when
$k$ tends to  $+\infty$) towards the same number $\lambda<0$. Applying to $f^{-N}$ the same argument we applied to $f^N$ for  proving the inequality (\ref{e.eta}), this gives that for any $\eta>0$, there exists $C>0$
such that for any $k,n\geq 0$  the following inequality holds:
\begin{equation}\label{e.etatilde}
\widetilde{S_k}(f^{-n.N}(x)) \leq (\lambda+\eta)k+2\eta n+2C.
\end{equation}
One concludes the claim (and hence the lemma) by noticing that, for every $0\leq k\leq n$, one has
$$
S_k(f^{-n.N}(x))= \widetilde{S_k}(f^{-(n-k).N}(x)),
$$
which implies
$$
0\leq {\max}_{0\leq k\leq n} S_k(f^{-n.N}(x))\leq 2\eta n+2C.
$$

\end{proof}

\subsection{Building the invariant manifolds}\label{ss.manifold}

In this section we build the local stable manifolds at the regular
points of an ergodic measure $\mu$, associated to a dominated splitting $E\oplus_{<} F$ on the
support of the measure $\mu$, under the assumption that the largest Lyapunov exponent $\lambda^+_E$ of $\mu$ in $E$ is negative.

We introduce a cone field on a neighborhood of $\supp(\mu)$: for any
$K>0$, there exists a continuous splitting $E'\oplus F'$ on a
neighborhood of $\supp(\mu)$ which allows us to define the cones
$$\cC^E_x=\{v=v_1+v_2\in T_xM=E'_x\oplus F'_x,\; \|v_2\|\leq K \|v_1\|\}.$$
Moreover, at any point $x\in \supp(\mu)$, we have $E_x\subset
\cC^E_x$.

Theorem~\ref{theo5} is a direct consequence of the next proposition:

\begin{proposition}\label{p.invariant-manifold}
Let $f$ be a $C^1$-diffeomorphism and $\mu$ be an
ergodic invariant probability measure whose support admits a
dominated splitting $E\oplus_{<} F$. Let $\lambda_E^+<\lambda^-_F$ be
the maximal Lyapunov exponent in $E$ and the minimal Lyapunov
exponent in $F$ of the measure $\mu$.

If $\lambda_E^+$ is strictly negative, then at $\mu$-almost every
point $x\in M$, there exists an injectively immersed $C^1$-manifold
$W^E(x)$ with $\dim W^E(x)=\dim E$, tangent to $E_x$, and which is a stable manifold: for any
$\lambda\leq 0$ contained in $(\lambda_E^+,\lambda_F^-)$ and
$\mu$-a.e. point $x$, we have
$$W^E(x)=\left\{y\in M,\; d(f^n(x),f^n(y)).e^{-\lambda.n}\underset{n\to +\infty}\longrightarrow 0\right\}.$$

Moreover, at $\mu$-a.e. point $x$ there exists a local manifold
$W^E_\loc(x)\subset W^E(x)$ satisfying:
\begin{enumerate}
\item\label{p.manifold1} $W^E_\loc(x)$ is an embedded $C^1$-disk centered at $x$, of radius $L^E(x)$ and tangent to $\cC^E$;
\item\label{p.manifold2} the sequence $\left(L^E(f^n(x))\right)_{n\in \ZZ}$ varies sub-exponentially;
\item\label{p.manifold3} $W^E(x)=\bigcup_{n\geq 0} f^{-n}(W^E_\loc(f^n(x)))$.
\end{enumerate}
\end{proposition}

The aim of Sections~\ref{ss.manifold} and \ref{ss.speed} is the proof of Proposition~\ref{p.invariant-manifold}. In this section (Section~\ref{ss.manifold}) we build the local stable manifolds $W_{loc}^E(x)$ and we prove the exponential decay of $d(f^n(x),f^n(y))$ for $y\in W_{loc}^E(x)$.  Section~\ref{ss.speed} ends the proof by showing that this exponential decay characterizes the points in the stable manifold.
In fact the proof does not use the adapted metric built in the previous section, but the function $A$
provided by Proposition~\ref{p.metric}.

Our main tool is the plaque family
theorem~\cite[theorem 5.5]{HPS} of Hirsch-Pugh-Shub:

\begin{theorem*}[Plaque family theorem, Hirsch-Pugh-Shub]
Let $f$ be a $C^1$-diffeomorphism and  $K$ be an $f$-invariant compact set admiting a dominated splitting $E\oplus_{<}
F$. Then, there exists a continuous family $(\widehat {D^E_x})_{x\in K}$
of embedded $C^1$-disks such that:
\begin{itemize}
\item for every $x\in K$, the disk $\widehat{D^E_x}$ is centered at $x$ and tangent to $E_x$;
\item the family $(\widehat{D^E_x})$ is locally invariant: there exists $\delta_0>0$ such that for each $x\in K$, the disk centered at $x$ of radius $\delta_0$
and contained in $\widehat{D^E_x}$ is mapped by $f$ into $\widehat
{D^E_{f(x)}}$.
\end{itemize}
\end{theorem*}

In order to prove Proposition~\ref{p.invariant-manifold}, we fix a
small positive constant $\varepsilon<-\frac  {\lambda^+_E} 3$.
In the previous section we obtained an integer $N$ and a measurable map  $A\geq 1$
associated to the bundle $E$ and to $\varepsilon$, which is well-defined on the set of $\mu$-regular points.
Let $C_f>1$ be a bound on the norm of the derivative $Df$.

One also chooses a small constant $\delta_1\in (0,\delta_0)$, so that for
any point $x\in \supp(\mu)$, any point $y\in \widehat{D_x^E}$ with
$d(x,y)<\delta_1$ and any vector $v\in T_y\widehat{D_x^E}$, we have
\begin{equation}\label{e.control-cone}
\|D_yf^N.v\|\leq e^{N.\varepsilon} \|Df_E^N(x)\|.\|v\|.
\end{equation}
Consider now $\delta>0$. For every $\mu$-regular point $x$  we denote by $D_x\subset\widehat{ D^E_x}$ the
disk centered at $x$ of radius  $L^E(x)=\delta/A(x)$. By choosing $\delta$
small enough, the disk $D_x$ is tangent to the cone field $\cC^E$.

The next lemma shows that, choosing $\delta>0$ small enough, the disk $D_x$ is contained in the stable manifold at $x$.
\begin{lemma}\label{l.growth-disk}
For $\delta$ smaller than $C^{-N}_f.\delta_1$, and $\mu$-a.e. point $x\in M$, each
forward iterate $f^n(D_x)$ of the disk $D_x$ is contained in the
corresponding disk $\widehat{D_{f^n(x)}^E}$ and has a diameter bounded by
$\delta_1$.

Moreover, the diameter $\diam(f^n(D_x))$ tend exponentially fast to $0$ when $n\to +\infty$; more precisely,  the sequence
$\left(\diam(f^n(D_x)).e^{-n(\lambda_E^++3\varepsilon)}\right)_{n\geq
0}$ goes to $0$ when $n$ tends to $+\infty$.
\end{lemma}
\begin{proof}

One proves the first part of the lemma inductively on $n$.
Let us assume that all the forward iterates
$f^m(D_x)$ up to an integer $n-1\geq 0$ are contained
in the corresponding disk $\widehat{D^E_{f^n(x)}}$ and have diameters
bounded by $\delta_1$. Since $\delta_1<\delta_0$
one first concludes  that the iterate $f^{n}(D_x)$ also is contained in the
disk $\widehat{D^E_{f^{n}(x)}}$. We will prove that its diameter also is
bounded by $\delta_1$.
Let $k \geq 0$ denote the largest integer such that $k.N\leq n$.
By the estimate~(\ref{e.control-cone}),
one deduces the following upper bound
$$
\diam\left(f^{n}(D_x)\right)\leq
C_f^{N}\;.\quad e^{k.N.\varepsilon}\prod_{0\leq \ell < k}
\left\|Df^N_{{|E}}(f^{\ell N}(x))\right\|\quad .\;\diam(D_x),
$$
Now by definition~(\ref{e.A}) of $A$ one deduces

\begin{equation}
\diam\left(f^{n}(D_x)\right) \leq C_f^{N}\;.\quad e^{k.N.(\lambda^+_E+2\varepsilon)}A(x)\quad .\;\frac{\delta}{A(x)}.
\label{e.growth-disk}
\end{equation}

By our choice of $\varepsilon$ and $\delta$, this gives as required:
\begin{equation}\label{e.growth}
\diam\left(f^{n}(D_x)\right)< e^{k.N.(\lambda^+_E+2\varepsilon)}\delta_1 \leq \delta_1.
\end{equation}

In order to get the second part of the lemma, one considers again
the estimate~(\ref{e.growth}) which has been now established
for all the forward iterates of $D_x$.
It implies:
$$\diam(f^n(D_x)).e^{-n(\lambda_E^++3\varepsilon)}
\leq e^{-(n-kN).(\lambda_E^++2\varepsilon)}.e^{-n\varepsilon}\delta_1\leq e^{-N.(\lambda_E^++2\varepsilon)}.e^{-n\varepsilon}\delta_1,$$
which goes to $0$ as $n\to +\infty$.
\end{proof}

\begin{coro}\label{c.speed} There is $\delta_2>0$ such that, for  every $\lambda>\lambda^+_E$, for $\mu$-almost every point $x\in M$,
for any point $y\in\widehat{D^E_x}$ one has

$$
\text{sup}_{n\geq 0} \; d(f^n(x),f^n(y))\leq \delta_2\Longrightarrow \left\{ \begin{array}{l}f^n(y)\in \widehat{D^E_{f^n(x)}} \mbox{ for all } n\geq 0,\\
\lim_{n\to+\infty}e^{-\lambda.n}d(f^n(x),f^n(y))=0.\end{array}\right.
$$
 \end{coro}
\begin{proof}
We choose $\delta_2\in(0,\delta_1)$ such that $\mu\{x\in M, L^E(x)>\delta_2\}>0$
Note that such a $\delta_2$ exists because $L^E$ is a positive
measurable map which is strictly positive at $\mu$-almost every point.

By definition of $\delta_0$, if $y\in \widehat{D_x^E}$ satisfies $d(x,y)\leq \delta_0$ then $f(y)\in \widehat{D^E_{f(x)}}$. As $\delta_2<\delta_0$ a simple inductive argument shows that, for every $x\in\supp(\mu)$ and every $y\in\widehat{D^E_x}$ one has
$$\text{sup}_{n\geq 0} \; d(f^n(x),f^n(y))\leq \delta_2\Longrightarrow f^n(y)\in \widehat{D^E_{f^n(x)}} \mbox{ for all } n\geq 0.$$

Now, the ergodicity of $\mu$ implies that,  for $\mu$-almost every point $x$, there are infinitely many $n>0$ for which $L^E(f^n(x))>\delta_2$,
implying that $f^n(y)\in D_{f^n(x)}$. Now Lemma~\ref{l.growth-disk}  implies that  $\left(d(f^n(x),f^n(y)).e^{-n(\lambda_E^++3\varepsilon)}\right)_{n\geq
0}$ goes to $0$ when $n$ tends to $+\infty$. In particular $d(f^n(x),f^n(y))$ tends to $0$.

For ending the proof, we fix now $\lambda>\lambda_E^+$.  We choose $\varepsilon_1\in (0,\frac{\lambda_E^+}3)$ such that $\lambda_E^++3\varepsilon_1<\lambda$.  This gives us a new function $A_1$ and a new function $L^E_1$ and thus a new family of disks $D_{x,1}\subset \widehat{D^E_x}$, and finally a new number $\delta_{2,1}$ such that $\mu\{x\in M, L^E_1(x)>\delta_{2,1}\}>0$. Notice that one may apply Lemma~\ref{l.growth-disk} to $\varepsilon_1$. Thus, the same argument as above proves that, for $\mu$-almost every $x$ and every $y\in\widehat{D^E_x}$ one has
$$\text{sup}_{n\geq 0} \; d(f^n(x),f^n(y))\leq \delta_2\Longrightarrow \lim_{n\to+\infty}e^{-\lambda.n}d(f^n(x),f^n(y))=0.$$

As a countable intersection of sets with $\mu$-measure equal to $1$ has measure equal to $1$, by choosing a sequence of $\lambda$ decreasing to $\lambda_E^+$ one gets that for $\mu$-almost every $x$, every $\lambda>\lambda_E^+$  and every $y\in\widehat{D^E_x}$ one has
$$\text{sup}_{n\geq 0} \; d(f^n(x),f^n(y))\leq \delta_2\Longrightarrow \lim_{n\to+\infty}e^{-\lambda.n}d(f^n(x),f^n(y))=0.$$
\end{proof}
As a direct corollary of Lemma~\ref{l.growth-disk} and
Corollary~\ref{c.speed} one gets:

\begin{coro}\label{c.speed2}For $\mu$-almost every point $x\in M$,
for any point $y\in D_x$,  for  every $\lambda>\lambda^+_E$ one has
$$
\lim_{n\to+\infty}e^{-\lambda.n}d(f^n(x),f^n(y))=0.
$$
\end{coro}

\subsection{Characterization of the invariant manifolds $W^E(x)$ by the speed of approximation}\label{ss.speed}

Lemma~\ref{l.speed} below  will end the proof of Proposition~\ref{p.invariant-manifold} (and therefore of Theorem~\ref{theo5}) by showing that the speed of approximation of $y\in D_x$ given by Corollary~\ref{c.speed2} provides a
characterization of the points in the stable manifold $W^E(x)$.

\begin{lemma}\label{l.speed} For $\mu$-almost every point $x$,  for every point $y\in M$ such that
$$
\lim_{n\to +\infty}  d(f^n(x),f^n(y))
=0 ,$$
we have the following dichotomy:
\begin{itemize}
\item either there is $n>0$ such that $f^n(y)\in D_{f^n(x)}$ (and so we have exponential convergence);
\item or for  every $\lambda\in (\lambda^+_E,\lambda^-_F)$ one has
$$
\lim_{n\to+\infty}e^{-\lambda.n}d(f^n(x),f^n(y))=+\infty.
$$

\end{itemize}
\end{lemma}

One now finishes the proof of Proposition~\ref{p.invariant-manifold}.
\begin{proof}[End of the proof of Proposition~\ref{p.invariant-manifold}]
By corollary~\ref{c.speed2}, for any
$\lambda>\lambda^+_E$, for $\mu$-a.e. point $x$ and any
$y\in D_x$, we have shown
\begin{equation}\label{e.converge}
d(f^n(x),f^n(y)).e^{-\lambda.n}\underset{n\to+\infty}\longrightarrow
0.
\end{equation}
With Lemma~\ref{l.slow} above, one obtains
the properties~(\ref{p.manifold1}) and~(\ref{p.manifold2}) of
Proposition~\ref{p.invariant-manifold} by defining
$W^E_\loc(x)=D_x$.

For some $\lambda\leq 0$ contained in $(\lambda^+_E,\lambda_F^-)$
and at $\mu$-a.e. point $x\in M$, one now considers any point $y\in
M$ which satisfies~(\ref{e.converge}). In particular
$d(f^n(x),f^n(y))$ goes to $0$ as $n$ goes to $+\infty$.
We are not in the second case
of Lemma~\ref{l.speed}, hence, there is a forward  iterate $f^n(y)$
of $y$ that belongs to $W^E_\loc(f^n(x))$. One deduces that the two
following sets coincide:
$$\left\{y\in M,\; d(f^n(x),f^n(y)).e^{-\lambda.n}\underset{n\to +\infty}\longrightarrow 0\right\}=\bigcup_{n\geq 0} f^{-n}\left(W^E_\loc(f^n(x))\right).$$
This set does not depend on the choice of $\lambda\leq 0$
in $(\lambda_E^+,\lambda_F^-)$ and will be denoted by $W^E(x)$.

Let us now remark that the forward iterates of any local manifold
$W^E_\loc(f^m(x))$ have a diameter which goes to $0$. Since  the
local manifolds at infinitely many iterates of $x$ have a radius
uniformly bounded away from zero, one deduces that any finite union
$\bigcup_{0 \leq k\leq m} f^{-k}\left(W^E_\loc(f^k(x))\right)$ is
contained in an embedded manifold
$f^{-n}\left(W^E_\loc(f^n(x))\right)$ for some large $n$. This implies
that $W^E(x)=\bigcup_{n\geq 0} f^{-n}\left(W^E_\loc(f^n(x))\right)$ is
an injectively immersed submanifold.
\end{proof} 

The end of the section will be devoted to the proof of Lemma~\ref{l.speed}.
We choose $\delta_{2}$ such that $\mu\{x\in M, L^E(x)>\delta_{2}\}>0$.
Hence for $\mu$-almost every point $x$, there exist infinitely many forward iterates $f^n(x)$ such that
$L^E(f^n(x))>\delta$.

We first consider some $\lambda\in (\lambda^+_E,\lambda^-_F)$.
The proof uses an invariant cone-field defined in a neighborhood of $\supp\mu$. More precisely
we will use the following classical result:
\begin{lemma}\label{l.conefield}
Consider any $\lambda' \in( \lambda,\lambda_F^-)$. Then there are
\begin{itemize}
\item  an integer $n_0>0$,
\item a neighborhood $U_0$ of $\supp(\mu)$
\item two  continuous bundles
$T_xM=E_0(x)\oplus F_0(x)$ for $x\in U_0$ such that $E_0(x)=E(x)$ and $F_0(x)=F(x)$ for $x\in\supp(\mu)$;
for every $K>0$ we denote by $\cC^F_K$ the cone field defined  for  $x\in U_0$ by
$$\cC^F_K(x)= \{v=v_1+v_2\in T_xM=E_0(x)\oplus F_0(x),\; \|v_1\|\leq K \|v_2\|\},$$
\item two positive numbers $0<b<a$ (hence, for every $x\in U_0$, one has $\cC^F_b(x)\subset\cC^F_a(x)$),
\end{itemize}

such that one has the following properties:
\begin{itemize}
\item for every $x\in U_0$ and every $v\in \cC^F_a(x)$ one has $\|Df^{n_0}(v)\|\geq e^{\lambda'.n_0}\|v\|$;
\item for every $x\in U_0$ such that $f^{n_0}(x)\in U_0$ one has $Df^{n_0}(\cC^F_a(x))\subset \cC^F_b(f^{n_0}(x))$.
\end{itemize}
\end{lemma}

We consider $r_0>0$ such that any two points $x,y\in M$ with $d(x,y)\leq r_0$ are joined by a unique geodesic segment of length bounded by $r_0$, which we denote by $[x,y]_{geo}$. Notice that the length $\ell([x,y]_{geo})$ is precisely $d(x,y)$.

\begin{lemma}\label{l.geodesic} Given any $r\in(0,r_0)$, there is a neighborhood $U_1\subset U_0$ of $\supp(\mu)$,  and $\delta_3\in(0,r) $ with the following property.

Consider $y,z\in U_1$ with $d(y,z)<\delta_3$, such that the segment $[y,z]_{geo}$ is contained in $U_0$ and is tangent to $\cC^F_a$. Then:

$$
f^{n_0}(y)\in U_1 \Longrightarrow
\left\{
\begin{array}{l}
d(f^{n_0}(y),f^{n_0}(z))<r,\\
{[f^{n_0}(y),f^{n_0}(z)]_{geo}}\subset U_1 \mbox{ and is tangent to } \cC^F_a,\\
d(f^{n_0}(y),f^{n_0}(z))\geq e^{\lambda.n_0} d(y,z).
\end{array}\right.
$$
\end{lemma}
\begin{proof}[Idea of the proof. ]
The proof of Lemma~\ref{l.geodesic} follows from the invariance of the conefield $\cC^F_a$ and from  the fact that, for $\delta_3$ small enough, the segment $f^{n_0}([y,z]_{geo})$ is very close (in the $C^1$-topology) to the geodesic segment $[f^{n_0}(y),f^{n_0}(z)]_{geo}$; in particular the  ratio $\frac{d(f^{n_0}(y),f^{n_0}(z))}{\ell(f^{n_0}([y,z]_{geo})}$, where $\ell(f^{n_0}([y,z]_{geo})$ is the length of the segment $f^{n_0}([y,z]_{geo})$, is almost $1$.
\end{proof}

Finally, the next lemma defines a kind of projection on $\widehat{D^E_x}$ of the points close enough to $x$:
\begin{lemma}\label{l.projection} For $r>0$ small enough, there is $C_1>0$ and $\delta_4\in(0,\delta_3)$
such that, for every  $x\in \supp(\mu)$, for every $y\in M$ with $d(x,y)\leq \delta_4$, one has:
\begin{itemize}
\item there is $z\in \widehat{D^E_x}$ such that $[y,z]_{geo}$ is tangent to $\cC^F_a$ and $d(y,z)<r$;
\item for every $z\in \widehat{D^E_x}$ such that $[y,z]_{geo}$ is tangent to $\cC^F_a$ and $d(y,z)<r$,  one has:
$$d(y,z)<C_1 d(x,y).$$
\end{itemize}
\end{lemma}
\begin{proof}[Idea of the proof.]The proof follows from the compactness of $\supp(\mu)$ and from the fact that the familly $\{\widehat{D^E_x}\}_{x\in\supp(\mu)}$ is a continuous family for the $C^1$ topology, hence is a compact family of $C^1$-disks.
\end{proof}

One can choose the constant $\delta_{4}>0$ small enough so that $e^{\lambda.n_{0}}.C_{1}.\delta_{4}<r$.
In particular, from Lemma~\ref{l.geodesic}, the segment $[y,z]_{geo}$ in Lemma~\ref{l.projection}
also satisfies $d(f^{n_{0}}(y), f^{n_{0}}(z))<r$.
 One also can always assume that $\delta_{3}+\delta_{4}<\lambda_{2}$.

\begin{proof}[Proof of Lemma~\ref{l.speed}]
Let us assume that the first case of the lemma does not occur and fix
some $\lambda\in (\lambda_{E}^+,\lambda_{F}^-)$.
One may choose an iterate $f^n(x)$ such that $L^E(f^n(x))>\delta_{2}$
and $d(f^{n+k}(x),f^{n+k}(y))<\delta_4$ for every $k\geq 0$. By Lemma~\ref{l.projection}, there is $z\in \widehat{D^E_x}$ such that the segment $[f^{n}(y),z]_{geo}$ is tangent to $\cC^F_a$ and has length bounded by
$\delta_3$. The distance $d(f^n(x),z)$ is thus less than $\lambda_{2}$ and $z$ belongs to $D_{f^n(x)}$.
Since we are not in the first case of the lemma, one has $f^n(y)\neq z$.
Hence, by Corollary~\ref{c.speed2} one has
$$\lim_{k\to+\infty}e^{- \lambda.k} d(f^k(z),f^{n+k}(x))=0.$$
One verifies by induction that, for every $k>0$, the segment $[f^{n+k}(y),f^k(z)]_{geo}$ is tangent to $\cC^F_a$ and has length bounded by $\delta_3$.
All the iterates $f^{n+k}(y)$ are contained in $U_1$, so Lemma~\ref{l.geodesic} implies that
$$\lim_{k\to+\infty}e^{-\lambda.k}d(f^{n+k}(y),f^k(z))=+\infty.$$
So $\lim_{k\to+\infty} e^{-\lambda.k}  d(f^{n+k}(y),f^{n+k}(x))=+\infty$, which gives the second case of the lemma as required.
\end{proof}
\section{Irregular Points}\label{irregular}

\subsection{Irregular$^+$ points of generic diffeomorphisms}

Theorem \ref{t.irregular} is a direct consequence of the two following results:

\begin{prop}\label{p.irregular} Let $p$ be a hyperbolic periodic saddle of a diffeomorphism $f$, whose homoclinic class
$H(p)$ is not reduced to the orbit of $p$ (i.e., $p$ has transverse homoclinic points). Let $K=\overline{W^s(\cO(p))}$ be the
closure of the stable manifold of the orbit of $p$. Then generic points in $K$ are irregular$^+$.
\end{prop}

\begin{prop}\label{p.Lyapirregular}
Let $p$ and $q$ be hyperbolic periodic saddles of a diffeomorphism $f$ which are
homoclinically related, and hence whose homoclinic classes $H(p)$ and $H(q)$ coincide. Assume that the largest Lyapunov
exponents of $p$ and $q$ are different. Let $K=\overline{W^s(\cO(p))}$ be the closure of the stable manifold of the orbit of $p$.
Then generic points in $K$ are Lyapunov-irregular$^+$.
\end{prop}

\begin{proof}[Proof of Theorem~\ref{t.irregular}]
Let $f$ be a $C^1$-generic tame diffeomorphism and $x$ be a generic point of $M$.  Assume that $\omega(x)$ is not a sink. Then, according to  \cite{MP,BC},  $\omega(x)$ is a non-trivial attracting homoclinic class $H(p)$, and hence $x$ belongs to the basin of this attracting class. Furthermore  $W^s(p)$ is dense in the open set $W^s(H(p))$. So $x$ belongs to the interior of closure $\overline{W^s(p)}$. As a consequence, a generic point $x$ of $M$ which belongs to $W^s(H(p))$  is a generic point in $\overline{W^s(p)}$.
As $H(p)$ is non-trivial and $f$ is generic, there is a hyperbolic periodic point $q\notin \cO(p)$ homoclinically related to $p$ such that the largest Lyapunov exponents of $p$ and $q$  are distinct.

Now, Propositions~\ref{p.irregular} and~\ref{p.Lyapirregular} imply that $x$ is irregular$^+$ and Lyapunov irregular$^+$, respectively.
\end{proof}

We now need only prove Propositions~\ref{p.irregular}
and~\ref{p.Lyapirregular}.

\subsection{Proofs of Propositions~\ref{p.irregular} and~\ref{p.Lyapirregular}}

The two propositions are consequences of three lemmas, the first
two of which are classical results from hyperbolic theory:

\begin{lemma}\label{l.mix}
If $K$ is a hyperbolic basic set there is $k \in \NN$ and a
compact subset $K_0\subset K$ such that $K$ is the disjoint union
$K_0\cup f(K_0)\cup\dots\cup f^{k-1}(K_0)$ and $K_0$ is invariant
by $f^k$ and is a topologically mixing basic set of $f^k$.
\end{lemma}

\begin{lemma}\label{l.stablebasic} Let $K$ be a topologically mixing hyperbolic basic set. Then for any point $z\in K$  one has
$$\overline{W^s(z)}=\overline{W^s(K)}.$$
\end{lemma}

The third lemma requires a proof:

\begin{lemma}\label{l.oscillation}
Let $K$ be a non-trivial hyperbolic basic set. Then generic points in $K$ are ir\-re\-gu\-lar$^+$.
\end{lemma}

\begin{proof}
Consider a continuous map $\phi\colon M\to \RR$ which equals $0$
on a periodic orbit $\gamma_0 \subset K$ and $3$ on a periodic
orbit $\gamma_1 \subset K$.

One denotes
$$O_n=\{x\in K| \exists m_1>n, \frac 1{m_1}\sum_{i=0}^{m_1-1}\phi(f^i(x))< 1 \mbox{ and } \exists m_2>n, \frac 1{m_2}\sum_{i=0}^{m_2-1}\phi(f^i(x))> 2\}.$$
$O_n$ is open in $K$ and one easily verifies that $O_n$ is dense.
For that one considers a fine Markov partition of $K$ and given
any point $z\in K$ one considers a point $x$ whose itinerary
coincides with that of $z$ an arbitrarily large number of periods
(so that the point $x$ is arbitrarily close to $z$), then with the
itinerary of $\gamma_1$ an arbitrarily large number of periods
(larger that $n$)(so that the average $\frac
1{m_1}\sum_{i=0}^{m_1-1}\phi(f^i(x))$ will be as close to $0$ as
we want), and next with the itinerary of $\gamma_2$ an arbitrarily
large number (so that the average  $\frac
1{m_2}\sum_{i=0}^{m_2-1}\phi(f^i(x))$ will be close to $3$).

Any point in the intersection  $R_0=\bigcap_0^\infty O_n$ is irregular$^+$.
\end{proof}

\begin{proof}[Proof of Proposition~\ref{p.irregular}]
Let $p$ be a periodic saddle point whose homoclinic class
$H(p)$ is not trivial and consider two periodic orbits
$\gamma_0\neq\gamma_1$ homoclinically related to $p$. We fix a
continuous function $\phi\colon M\to \RR$ such that
$\phi(\gamma_0)=0$ and $\phi(\gamma_1)=3$.

We define $$W_n=\{x\in \overline{W^s(\cO((p))}| \exists m_1>n,
\frac 1{m_1}\sum_{i=0}^{m_1-1}\phi(f^i(x))< 1 \mbox{ and } \exists
m_2>n, \frac 1{m_2}\sum_{i=0}^{m_2-1}\phi(f^i(x))> 2\}.$$ Again,
this is an open set of $\overline{W^s(\cO(p))}$ and any point
in $G_0=\bigcap_0^\infty W_n$ is irregular$^+$. It remains to
prove that $W_n$ is dense.

We consider a hyperbolic basic set $K \subset H(p)$
containing $\gamma_0$ and $\gamma_1$. Let $x$ be a point in the
residual subset $R_0$ built in the proof of
Lemma~\ref{l.oscillation}. Notice that the stable manifold
$W^s(f^i(x))$ is contained in $W_n$ for every $n>0$ and every
$i\in \ZZ$.   By lemmas~\ref{l.mix} and~\ref{l.stablebasic}, the stable manifold of
the orbit of $x$ is dense in $\overline{W^s(\cO(p))}$. So the
open sets $W_n$ are dense and $G_0$ is residual, concluding the
proof of Proposition~\ref{p.irregular}.

\end{proof}

The proof of Proposition~\ref{p.Lyapirregular} is more delicate
because, a priori, points in the stable manifold of a Lyapunov
irregular$^+$ point may be Lyapunov regular. For this reason we
will follow a more subtle strategy.

\begin{proof}[Proof of Proposition~\ref{p.Lyapirregular}]
Let $p$ be a periodic saddle point point whose homoclinic class
$H(p)$ is not trivial and consider a periodic point $q$
homoclinically related to $p$. We assume that $p$ and $q$ have
distinct largest Lyapunov exponents $0<\lambda_p <\alpha <\beta
<\lambda_q$, for some positive numbers $\alpha,\beta$.

We define $$U_n=\{x\in \overline{W^s(\cO(p))}| \exists m_1,m_2>n, \frac 1{m_1}\log \|D(f^{m_1}(x)\|< \alpha \mbox{ and }  \frac 1{m_2}\log \|D(f^{m_2}(x)\|> \beta\}.$$

The set $U_n$ is open and any point in $G_1=\bigcap_0^\infty U_n$
is Lyapunov irregular$^+$. In order to prove Proposition~\ref{p.Lyapirregular} it suffices to prove that the
$U_n$ are dense in $\overline{W^s(\cO(p))}$. For that we consider
a basic set $K$ containing the orbits of $p$ and $q$ and we will
prove
\begin{claim}
There is  a  point  $x\in K$ whose stable manifold $W^s(x)$ is contained in $U_n$.
\end{claim}

Now Lemmas~\ref{l.mix} and~\ref{l.stablebasic} implies that $U_n$ is dense in
$\overline{W^s(\cO(p))}$, concluding the proof of
Proposition~\ref{p.Lyapirregular}. It remains to prove the claim.
\end{proof}

\begin{proof}[Proof of the claim]

We fix a Markov partition generating $K$ and we denote by $a$ and
$b$ the itineraries of $p$ and $q$ in terms of this partition. We
denote by $T_{pq}$ (resp. $T_{qp}$) some itinerary from the
rectangle containing $p$ (resp. $q$) to the rectangle containing
$q$ (resp. $p$). We denote by $\ell(m)$ the length of a word $m$.

We consider the positively infinite word $$a^{t_0}T_{pq}b^{t_1}
T_{qp}a^{t_2}\dots a^{t_i}T_{pq}b^{t_{i+1}}
T_{qp}a^{t_{i+2}}\dots,$$ with

$$\lim_{i\to +\infty} \frac{t_i}{\sum_{j=0}^{i-1} t_j}=+\infty.$$

Let $y$ be a point of $K$ having this itinerary and $z\in W^s(y)$.
We will show that $z\in U_n$. After some time, the point $z$ has
the same itinerary as the point $y$. We look at the successive
stays of the orbit of $z$ close to $p$ and $q$.

We fix a non-decreasing sequence of positive integers $k_n\to
+\infty$ satisfying $\lim_{n\to \infty} \frac{k_n}n=0.$

We consider the point $z_{i}= f^{r_i}(z)$ with
$$r_{2i}=k_{2i}.\ell(a)\;+\; \ell(a).\sum_0^{i-1}t_{2j}\;+\;
\ell(b).\sum_0^{i-1}t_{2j+1}\; +\;i.\left(\ell(T_{pq})+\ell(T_{qp})\right)$$
and
$$r_{2i+1}=k_{2i+1}.\ell(b)\;+\; \ell(a).\sum_0^{i}t_{2j}\;+\;
\ell(b).\sum_0^{i-1}t_{2j+1}\;+\;T_{pq}\;+\;i.\left(\ell(T_{pq})+\ell(T_{qp})\right).$$

One easily verifies that for any neighborhoods $U_p$ of the orbit
of $p$ and $U_q$ of the orbit of $q$, there is $i_0$ such that for
$i\geq i_0$ one has :
$$z_{2i},f(z_{2i}), \dots, f^{(t_{2i}-k_{2i})\cdot \ell(a)}(z_{2i})\in U_p, $$
and
$$z_{2i+1},f(z_{2i+1}), \dots, f^{\left(t_{2i+1}-k_{2i+1}\right)\cdot \ell(b)}(z_{2i+1})\in U_q .$$

We write $s_{2i}=(t_{2i}-k_{2i})\cdot \ell(a)$ and  $s_{2i+1}=
\left(t_{2i+1}-k_{2i+1}\right)\cdot \ell(b)$.

One deduces that
$$\lim_{i\to +\infty}\frac1{s_{2i}}\log\|Df^{s_{2i}}(z_{2i})\|= \lambda_p<\alpha\mbox{, and } \lim_{i\to +\infty}\frac1{s_{2i+1}}\log\|Df^{s_{2i+1}}(z_{2i+1})\|=\lambda_q>\beta.$$

Furthermore, $\lim_{i\to \infty} \frac{s_i}{r_i}=+\infty$. As a
consequence, one obtains that the norms  $\|Df^{r_i}(z)\|$ and
$\|(Df^{r_i}(z))^{-1}\|$ are very small in comparison with
$(\frac\alpha{\lambda_p})^{s_i}$ and
$(\frac{\lambda_q}\beta)^{s_i}$. One deduces that, for $i$ large
enough, one has

$$\frac1{r_{2i}+s_{2i}}\log\|Df^{r_{2i}+s_{2i}}(z)\|<\alpha \mbox{, and }\frac1{r_{2i+1}+s_{2i+1}}\log\|Df^{r_{2i+1}+s_{2i+1}}(z)\|>\beta.$$

We proved $z\in U_n$ for every $n$, concluding the proof of the claim.
\end{proof}

\subsection{Generic points of generic diffeomorphisms are irregular}

Let $f$ be a $C^1$-generic diffeomorphism: by~\cite{BC} the chain-recurrent set
coincides with the non-wandering set of $f$;
moreover for each connected component $U$ of $\interior \Omega(f)$,
there exists a periodic orbit $\cO$
whose homoclinic class is non-trivial and such that the closure $K$ of $W^s(\cO)$
contains $U$.
Let us consider a generic point $x\in M$. Two cases occurs.
\begin{itemize}
\item Either $x$ belongs to $M\setminus \Omega(f)$ and in this case,
it is non-recurrent.
So Conley theory~\cite{Con} implies that the omega- and the alpha-limit sets of $x$
are contained in different chain recurrence classes of $f$ which are disjoint compact sets.
This implies that $x$ is irregular:
the positive and the negative averages along the orbit of $x$ of a continuous map
$\varphi$ with $\varphi(\alpha(x))=0$ and $\varphi(\omega(x))=1$ will converge to $0$
and $1$ respectively.

\item Or $x$ belongs to $\interior(\Omega(f))$ and is generic in the closure $K$ of
$W^s(\cO)$ for a periodic orbit $\cO$ having a non-trivial homoclinic class.
By proposition~\ref{p.irregular}, such a point $x$ is irregular$^+$.
\end{itemize}


\vspace{10pt}

\noindent \textbf{Flavio Abdenur} (flavio@mat.puc-rio.br)

\noindent Departamento de Matem\'atica, PUC-Rio de Janeiro

\noindent 22460-010 Rio de Janeiro RJ, Brazil

\vspace{10pt}

\noindent \textbf{Christian Bonatti} (bonatti@u-bourgogne.fr)

\noindent  CNRS - Institut de Math\'ematiques de Bourgogne, UMR 5584

\noindent  BP 47 870,  21078 Dijon Cedex, France

\vspace{10pt}

\noindent \textbf{Sylvain Crovisier}
(crovisie@math.univ-paris13.fr)

\noindent  CNRS - LAGA, UMR 7539, Universit\'e Paris 13

\noindent  99, Av. J.-B. Cl\'ement, 93430 Villetaneuse, France
\end{document}